\documentclass[12pt,twoside]{amsart}
\usepackage[margin=3cm]{geometry}
\usepackage[colorlinks=false]{hyperref}
\usepackage[english]{babel}
\usepackage{graphicx,titling}
\usepackage{float}
\usepackage{amsmath,amsfonts,amssymb,amsthm}
\usepackage{lipsum}
\usepackage[T1]{fontenc}
\usepackage{fourier}
\usepackage{color}
\usepackage[latin1]{inputenc}
\usepackage{esint}
\usepackage{caption}
\usepackage{ccicons}

\makeatletter
\def\blfootnote{\xdef\@thefnmark{}\@footnotetext}
\makeatother

\newcommand\ccnote{
    \blfootnote{\copyright\,\, Mark Allen, Dennis Kriventsov, and Robin Neumayer}
    \blfootnote{\ccLogo\, \ccAttribution\,\, Licensed under a \href{https://creativecommons.org/licenses/by/4.0/}{Creative Commons Attribution License (CC-BY)}.}
}

\usepackage[export]{adjustbox}
\numberwithin{equation}{section}
\usepackage{setspace}
\renewcommand{\le}{\leqslant}
\renewcommand{\leq}{\leqslant}

\renewcommand{\geq}{\geqslant}
\renewcommand{\mathbb}{\varmathbb}
\usepackage{fancyhdr}
\newtheorem{theorem}{Theorem}[section]
\newtheorem{lemma}[theorem]{Lemma}
\newtheorem{corollary}[theorem]{Corollary}
\newtheorem{proposition}[theorem]{Proposition}
\newtheorem{definition}[theorem]{Definition}
\newtheorem{remark}[theorem]{Remark}
\usepackage{esint}

\newcommand{\R}{{\mathbb{R}}}
\newcommand{\e}{{\epsilon}}
\newcommand{\g}{{\gamma}}

\newcommand{\pa}{{\partial}}

\newcommand{\vphi}{{\varphi}}
\def\Id{\text{Id}}
\def\divv{\text{div}}

\newcommand{\vertiii}[1]{{\left\vert\kern-0.25ex\left\vert\kern-0.25ex\left\vert #1 
    \right\vert\kern-0.25ex\right\vert\kern-0.25ex\right\vert}}

\newcommand{\vol}{{\rm{vol}}}

\newcommand{\Om}{\Omega}
\newcommand{\om}{\omega}
\newcommand{\na}{{\nabla}}

\def\b{\beta}
\def\e{\varepsilon}

\def\r{\rho}

\def\w{\omega}

\def\g{\gamma}

\def\sm{\setminus}

\def\8{\infty}

\newcommand{\B}{B}
\newcommand{\BB}{{\mathcal{B}}}

\newcommand{\tor}{{\text{tor}}}

\address{Mark Allen, Brigham Young University, Department of Mathematics, Provo, UT}
\email{allen@math.byu.edu}
\medskip
\address{Dennis Kriventsov, Rutgers University, Department of Mathematics, Piscataway, NJ} 
\email{dnk34@math.rutgers.edu}
\medskip
\address{Robin Neumayer, Carnegie Mellon University, Department of Mathematics, Pittsburgh, PA}
\email{neumayer@cmu.edu}

\begin{document}

\thispagestyle{empty}

\begin{minipage}{0.28\textwidth}
\begin{figure}[H]
\includegraphics[width=2.5cm,height=2.5cm,left]{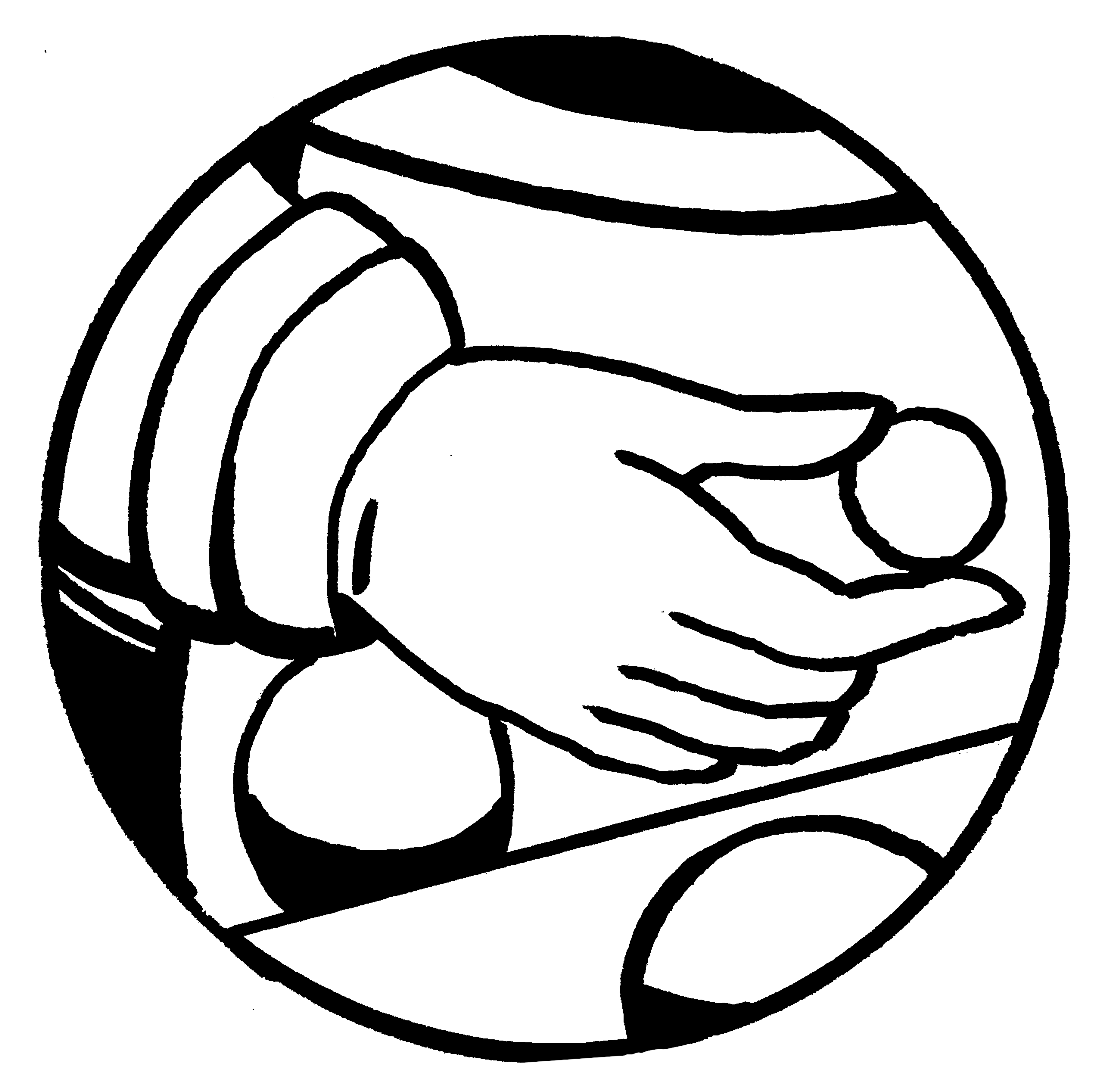}
\end{figure}
\end{minipage}
\begin{minipage}{0.7\textwidth} 
\begin{flushright}
Ars Inveniendi Analytica (2023), Paper No. 1, 49 pp.
\\
DOI 10.15781/e7f3-a487
\\
ISSN: 2769-8505
\end{flushright}
\end{minipage}

\ccnote

\vspace{1cm}


\begin{center}
\begin{huge}
\textit{Sharp Quantitative Faber-Krahn Inequalities and the}

\textit{Alt-Caffarelli-Friedman \\ Monotonicity Formula}

\end{huge}
\end{center}

\vspace{1cm}


\begin{minipage}[t]{.28\textwidth}
\begin{center}
{\large{\bf{Mark Allen}}} \\
\vskip0.15cm
\footnotesize{Brigham Young University}
\end{center}
\end{minipage}
\hfill
\noindent
\begin{minipage}[t]{.28\textwidth}
\begin{center}
{\large{\bf{Dennis Kriventsov}}} \\
\vskip0.15cm
\footnotesize{Rutgers University}
\end{center}
\end{minipage}
\hfill
\noindent
\begin{minipage}[t]{.28\textwidth}
\begin{center}
{\large{\bf{Robin Neumayer}}} \\
\vskip0.15cm
\footnotesize{Carnegie Mellon University} 
\end{center}
\end{minipage}

\vspace{1cm}


\begin{center}
\noindent \em{Communicated by Guido De Philippis}
\end{center}
\vspace{1cm}


\noindent \textbf{Abstract.} \textit{The objective of this paper is two-fold. First, we establish new sharp quantitative estimates for Faber-Krahn inequalities on simply connected space forms. In these spaces, geodesic balls uniquely minimize the first  eigenvalue of the Dirichlet Laplacian among all sets of a fixed volume. We prove that the gap between the first eigenvalue of a given set $\Omega$ and that of the ball quantitatively controls both the $L^1$ distance of this set from a ball {\it and} the $L^2$ distance between the corresponding eigenfunctions:
\[
\lambda_1(\Omega) - \lambda_1(B) \gtrsim |\Omega \Delta B|^2 + \int |u_{\Omega} -u_B|^2,
\] 
where $B$ denotes the nearest geodesic ball to $\Omega$ with $|B|=|\Omega|$ and $u_\Omega$ denotes the first eigenfunction with suitable normalization. On Euclidean space, this extends a result of Brasco-De Phillipis-Velichkov; the eigenfunction control largely builds upon new regularity results for minimizers of critically perturbed Alt-Cafarelli type functionals in our companion paper. On the round sphere and hyperbolic space, the present results are the first sharp quantitative results with respect to any distance; here the local portion of the analysis is based on new implicit spectral analysis techniques. \\
Second, we apply these sharp quantitative Faber-Krahn inequalities in order to establish a quantitative form of the Alt-Caffarelli-Friedman (ACF) monotonicity formula. A powerful tool in the study of free boundary problems, the ACF monotonicity formula is nonincreasing with respect to its scaling parameter for any pair of admissible subharmonic functions, and is constant if and only if the pair comprises two linear functions truncated to complementary half planes. We show that the energy drop in the ACF monotonicity formula from one scale to the next controls how close a pair of admissible functions is from  a pair of complementary half-plane solutions. In particular, when the square root of the energy drop summed over all scales is small, our result implies the existence of tangents (unique blowups) of these functions.}
\vskip0.3cm

\vspace{0.5cm}


\section{Introduction}
Let $(M,g)$ be a smooth Riemannian $n$-manifold, and let $\Omega $ be an open bounded subset of $M$. The first Dirichlet eigenvalue, or principal frequency, of $\Omega$ is defined by
\begin{equation}\label{eqn: lambda 1a}
\lambda_1(\Omega) = \inf \left\{ \frac{\int_\Omega |\na u|^2 }{\int_\Omega u^2} \ :\ u \in C^{\infty}_0(\Omega) \right\}.
\end{equation}
Here and in the sequel, $|\na u|^2 = \langle \nabla_g u, \nabla_g u\rangle_g$ and integration is with respect to the volume measure induced by the metric $g$.
 The infimum in \eqref{eqn: lambda 1a} is achieved, and a minimizer $u_\Omega$ satisfies the equation
\begin{equation}\label{eqn: eval eqn 1}
	\begin{cases}
-\Delta u_\Omega =\lambda_1(\Omega)\, u_\Omega & \text{ in } \Omega\\ 
 u_\Omega =0 & \text{ on } \pa \Omega,
 \end{cases}
\end{equation}
where $\Delta =\Delta_g$ is the Laplacian with respect to $g$. Such a function $u_\Om$ is called a first Dirichlet eigenfunction of $\Om$ 
and $\lambda_1(\Om)$ is the smallest number for which a nontrivial solution to this eigenvalue problem exists. If $\Om$ is connected, then the first Dirichlet eigenfunction is unique up to constant multiples.

In Euclidean space, balls have the smallest principal frequency among  subsets $\Omega \subset \R^n$ of a given volume. This fact, first established by Faber \cite{Faber23} and Krahn \cite{Krahn25} in the 1920s, readily follows from the Polya-Szeg\"{o} principle \cite{PSBook}, which
 relies only on the coarea formula and the isoperimetry of balls. On all simply connected space forms (Euclidean space, hyperbolic space and the round sphere), geodesic balls are isoperimetric sets for every volume. Consequently, on simply connected space forms, the Polya-Szeg\"{o} principle holds  and geodesic balls minimize  the principal frequency among sets of a given volume. 
The resulting inequality is known as the {\it Faber-Krahn inequality}: letting $\B $ denote a geodesic ball and $|\cdot |$ denote the volume measure induced by the metric on either Euclidean space, hyperbolic space, or the round sphere, we have 
\[
\lambda_1(\Omega) \geq \lambda_1(\B) \qquad \text{ whenever } |\Omega| = |\B|.
\]
It is known, moreover, that equality is achieved if and only if $\Omega$ is a geodesic ball, up to a set of capacity zero. On the round sphere, the Faber-Krahn inequality was first established by Sperner in \cite{s73} and is often called Sperner's inequality.
We refer the reader to \cite{AshBen07, Baernstein, ChavelBook} for a discussion of spectral inequalities and symmetrization techniques in these settings. \\

Our first main theorem is a sharp quantitative stability result for the Faber-Krahn inequality on simply connected space forms.  Broadly speaking, a quantitative stability result in this context states that 
\begin{equation}\label{eqn: stability general form}
\lambda_1(\Omega) -\lambda_1(\B) \geq c \text{ dist}(\Omega, \B)^\alpha
\end{equation}
for some $\alpha >0$, where $\B$ is a geodesic ball with $|\Om| = |\B|$ and $\text{dist}(\Omega, \B)$ is a suitable notion of distance between $\Omega$ and the nearest ball $\B$. 
In the statement below, we let $u_\Omega$ denote any first eigenfunction of $\Omega$, i.e. solution of \eqref{eqn: eval eqn 1}, that is normalized so that $u_\Omega \geq 0$ and $\int u_\Omega^2 = 1$, and extended by $0$ to be defined on the entire space. We let $A\Delta B= (A\setminus B) \cup (B\setminus A)$ denote the symmetric difference between sets.

\begin{theorem}[Sharp quantitative stability  for Faber-Krahn inequalities]\label{thm: quantitative FK general}\label{thm: quantitative sperner}
Fix $n\geq 2$ and let $(M^n,g)$ denote the round sphere, Euclidean space, or hyperbolic space. For any $v>0$ (with $v<|S^n|$ in the case of the round sphere), there exists a constant $c=c(n,v)$ such that the following holds. For any open bounded set $\Omega\subset M$ with $|\Omega|=v,$  we have 
\begin{equation}\label{eqn: stability statement}
	\lambda_1(\Omega) - \lambda_1(\B) \geq c \inf_{x \in M}\left( \left| \Omega \Delta \B(x)\right|^2 + \int_{M} \left| u_\Omega - u_{\B(x)}\right|^2 \right)\,.
\end{equation}
Here $\B(x)$ is a geodesic ball centered at $x$ with radius uniquely chosen so that $|\B(x)| =v$, and $\B$ denotes such a ball centered at any point.
\end{theorem}
\begin{remark}[Optimality of the quadratic power]
	{\rm
	Considering ellipsoidal perturbations of a geodesic ball (in normal coordinates in the case of the sphere or hyperbolic space) shows that the quadratic power $\alpha =2$  is optimal for both of the terms on the right-hand side of \eqref{eqn: stability statement}, in the sense that the result is false if $2$ is replaced by any $\alpha <2.$
	}
\end{remark}

On Euclidean space, quantitative stability for the Faber-Krahn inequality was first considered by   Hansen and Nadirashvili \cite{HansenNadir94} and Melas \cite{Melas}. Both of these papers restrict their attention to convex sets or simply connected planar sets, and consider ``$L^\infty$'' (Hausdorff) type set distances $\text{dist}(\Omega,\B)$. Among general open bounded sets, simple examples show that even qualitative stability fails with respect to such a distance. For this reason, the {\it asymmetry}, $\text{dist}(\Om, \B) =\inf_{x \in \R^n} |\Omega \Delta \B(x)|$, is a more natural notion of set distance to consider, and quantitative stability with respect to the asymmetry was conjectured in \cite{BhatWeit} and \cite{Nadir97}. Non-sharp versions of this conjecture (i.e. with $\alpha >2$ in \eqref{eqn: stability general form}) were established in following years in  \cite{Bhat01} and \cite{FMP09}. All of these results are based on applying quantitative forms of the isoperimetric inequality to level sets of eigenfunctions.

The first sharp quantitative Faber-Krahn inequality on Euclidean space was established by  Brasco, De Philippis, and Velichkov \cite{BDV15}, stating 
\begin{equation}\label{eqn: BDV}
\lambda_1(\Omega) - \lambda_1(\B) \geq c \inf_{x \in \R^n} |\Omega \Delta \B(x)|^2
\end{equation}
for all open bounded sets $\Omega$, where $|\B(x)| = |\Om|$.  Their proof was based on a selection principle approach  introduced by Cicalese-Leonardi \cite{CiLe12} (see also \cite{AFM13}) using the  regularity theory for Alt-Caffarelli type functionals. The selection principle is also the basis of the proof of Theorem~\ref{thm: quantitative FK general}, and we discuss these proofs further in Section~\ref{ss: proof talk} below. Following the work of \cite{BDV15}, similar methods were used to establish sharp quantitative stability for the $p$-Laplacian \cite{FuscoZhang} as well as for the capacity  \cite{DPMMCapacity} and $p$-capacity \cite{Mukoseeva+2021} inequalities.
We refer the reader to \cite{BrascoDePSurvey} for an overview on spectral inequalities in quantitative form.

A shortfall of the asymmetry as a measure of the distance between sets is that it only detects modifications up to sets of measure zero, while the eigenvalue sees modifications of a set up to sets of zero capacity. For instance, if we let $\Omega = B(0,1) \setminus\{x_1 =0\}$ be a slit domain in $\R^2$, then the left-hand side of \eqref{eqn: BDV} is strictly positive while the right-hand side is zero. In contrast, the right-hand side of \eqref{eqn: stability statement} is strictly positive in this example (and more generally, whenever the left-hand side is), and in this way, Theorem~\ref{thm: quantitative FK general} nontrivially strengthens the result of \cite{BDV15} even in Euclidean space. This issue is discussed in the survey paper \cite{BrascoDePSurvey}, and in particular Theorem~\ref{thm: quantitative FK general} resolves \cite[Open Problem 3]{BrascoDePSurvey}.

On the sphere and hyperbolic space, significantly less progress has been made toward quantitative Faber-Krahn inequalities. These spaces lack the scaling invariance of Euclidean space, and certain computations that can be carried out explicitly on Euclidean space cannot be done in these spaces. In fact, the explicit forms of the first eigenvalue and eigenfunction of geodesic balls are not known in closed form!  \'{A}vila established analogues of Melas' aforementioned results for convex subsets of $S^2$ and $H^2$ in \cite{Avila02}, following quantitative estimates for other spectral inequalities in \cite{Xu95}.
 Qualitative stability results for the Faber-Krahn inequality and other spectral inequalities on space forms is established in \cite{ABC09}. See also \cite{sphere, hyperbolic} for quantitative isoperimetric inequalities on space forms.
To our knowledge, Theorem~\ref{thm: quantitative FK general} is the first (sharp or nonsharp) quantitative stability result for Faber-Krahn inequalities on the round sphere and hyperbolic space that is valid for all open bounded sets $\Om$.\\

\medskip

It has been understood in recent years that  eigenvalue optimization problems are closely linked to the study of free boundary problems. When optimizing functionals involving (linear combinations of) higher eigenvalues, or of the first eigenvalue with respect to additional constraints, one can no longer expect to explicitly characterize minimizers.
 Instead, one hopes to show that a minimizer exists in a weak class of sets, and then show that the boundary of a minimizer is regular at most points. The paper \cite{BL09} considered the minimization of $\lambda_1(\Omega)$ among subsets $\Omega$ of some fixed bounded open container. Here the authors established regularity properties of minimizing sets by using  the regularity theory of Alt and Caffarelli \cite{AC81} for the one-phase free boundary problem. More recently, minimization problems for functionals that are  linear combinations of Dirichlet eigenvalues have been recast as vectorial free boundary problems in \cite{KL1, KL2, MTV17, CSY}, and regularity of minimizers has been studied through this lens of free boundary regularity.

In these examples, regularity theory for free boundary problems is used as a tool in the study of eigenvalue optimization problems. In the next main result of this paper, the stream of tools will go in the opposite direction: we apply the  quantitative Faber-Krahn inequality of Theorem~\ref{thm: quantitative FK general} on the sphere in order to establish new quantitative estimates for the  Alt-Caffarelli-Friedman monotonicity formula, a powerful tool in the study of various types of free boundary problems.

Fix $n\geq 2 $ and let $B_r = B(0,r)\subset \R^n$ denote the ball of radius $r$ centered at the origin in Euclidean space. Let $u_1,u_2\in H^1(B_2)$ be two nonnegative continuous functions satisfying
\begin{equation}\label{eqn: harmonic functions} 
\begin{cases}
	\Delta u_i \geq 0&   \ \ \text{ in } \ \ \{u_i>0\}=:\Om_i\\
	u_1\cdot u_2 = 0 & \  \ \text{ in } B_2.
\end{cases}
\end{equation}
The Alt-Caffarelli-Friedman monotonicity formula (hereafter denoted ACF formula) $J:(0,1) \to \R$ is defined by 
\begin{equation}\label{eqn: ACF 2d}
J(r) = J[u_1, u_2](r) = \frac{1}{r^4} \int_{ B_r} |\na u_1|^2 \int_{ B_r} |\na u_2|^2\, 
\end{equation}
when $n= 2$ and
\begin{equation}\label{eqn: ACF higher d}
J(r) =  J[u_1, u_2](r) =\frac{1}{r^4} \int_{ B_r} \frac{|\na u_1|^2}{|x|^{n-2}} \int_{ B_r} \frac{|\na u_2|^2}{|x|^{n-2}}
\end{equation}
for $n\geq 3$. Heuristically thinking, one can view $J(r)$ as the product of the average Dirichlet energies of $u_1, u_2$ on $B_r$. This quantity has the important property of being nondecreasing for $r \in (0,1)$. Note that this in particular guarantees the existence of the limit 
\[
J(0^+) := \lim_{r\to 0} J(r).
\]
Furthermore, if $J(0^+)$ is strictly positive, then $J(r)$ is constant with respect to $r$ if and only if
\[
u_1 = \beta_1 (x\cdot \nu)^+ ,\qquad \qquad u_2 = \beta_2(x \cdot \nu)^-,
\]
for some $\nu \in S^{n-1}$ and $\beta_1, \beta_2> 0$. Here $f^+$ and $f^-$ respectively denote the positive and negative parts of a function $f$. In other words, the ACF monotonicity formula is constant and positive if and only if $u_1, u_2$ are linear functions defined on complementary half spaces $\Omega_1$ and $\Omega_2. $

Since its introduction in \cite{acf84}, the ACF monotonicity formula has proven useful in both free boundary problems and harmonic analysis. One application of the ACF monotonicity formula is the following. Suppose that a pair of functions $u_1, u_2$ satisfying \eqref{eqn: harmonic functions} have $J(0^+)>0$, and consider any sequence $r_{k} \to 0$. There exists a subsequence (which we do not relabel), constants $\beta_1, \beta_2>0$ and a direction $\nu \in S^{n-1}$ such that
\[
 \frac{u_1(r_{k} x)}{r_{k}} \to \beta_1(x \cdot \nu)^+ , \qquad \qquad \frac{u_2(r_{k} x)}{r_{k}} \to \beta_2(x \cdot \nu)^-.
 \]
 
 However, it is important to note that these blowup limits depend on the sequence $r_{k}$ and  {\it does not} imply that $u_1$ and $u_2$ or their interfaces $\Omega_1, \Omega_2$ have unique tangents at the origin. In fact, this need not be the case!  In \cite{ak20} the first two authors gave an example of two harmonic functions defined on complementary subsets of $B_2$, which have a slowly spiraling interface,  where different sequences $r_k \to 0$ result in different 
$\nu$ in the limit.

Toward understanding conditions that {\it do} guarantee the existence of unique tangents for functions $u_1$ and $u_2$ as in \eqref{eqn: harmonic functions}, we establish the following quantitative estimates for the Alt-Caffarelli-Friedman monotonicity formula. Roughly speaking, this theorem says that if $J(1) -J(0^+)$ is small, then $u_1$ and $u_2$ are quantitatively close to linear functions $\beta_i (x\cdot\nu)^{\pm}$.  Note that if $J(0^+)$ is strictly positive, then up to rescaling we may assume that $J(0^+)=1$.
 \begin{theorem}[Sharp quantitative stability for the ACF monotonicity formula]   \label{t:stability}
Fix $n\geq 2$ and let $\rho \in [0,1/2]$. There exists a constant $C=C(n)>0$ such that the following holds. Suppose that $u_1,u_2$ satisfy \eqref{eqn: harmonic functions}.  Then there exist constants $\beta_1, \beta_2>0$ and $\nu \in S^{n-1}$ such that 
 \[
  \int_{B_1 \setminus B_{\r}} \left(u_1 -\beta_1 (x\cdot \nu)^+ \right)^2 + \left(u_2 -\beta_2 (x\cdot \nu)^-\right)^2
  \leq C\log\left(\frac{J(1)}{J(\r)}\right) \sum_{i=1}^2 \| u_i \|_{L^{2}(B_1)}^2.  
 \] 
 Furthermore, there exists a constant $\epsilon_0=\epsilon_0(n)>0$ such that if  the quotient
 $\log(J(1)/J(0^+))< \epsilon_0$, then $\beta_i$ and $\nu$ may be chosen independently of $\rho \in [0,1/2]$. 
\end{theorem}

This theorem is sharp, in the sense that the dependence on $\log \frac{J(1)}{J(\r)}$ on the right cannot be improved; this can be seen from the examples constructed in \cite[Section 2]{ak20}. 

A direct consequence of Theorem~\ref{t:stability} is the following corollary, which gives a criterion for naturally scaled blow-ups of the functions $u_i$ at $0$ to be unique. We believe this criterion is essentially sharp, for the same reasons that Theorem \ref{t:stability} is sharp. This resolves a question raised in \cite{ak20}.
\begin{corollary}[Uniqueness of blowups]\label{c:uniqueblowup}
	Fix $n\geq 2$. Suppose that $u_1,u_2$ satisfy \eqref{eqn: harmonic functions}, assume $J(0^+) = 1$, and let $\w(r) = J(r) - 1$. Assume in addition that
	\[
	\int_0^1 \frac{\sqrt{\w(r)}}{r} dr < \infty.
	\]
	Then there exists a number $\b_1 > 0$ and a vector $\nu \in S^{n-1}$ such that
	\[
	\lim_{r \searrow 0} \frac{1}{r^{n + 2}} \int_{B_{r}} |u_1 - \b_1 (\nu \cdot x)_+|^2 = 0.
	\]
\end{corollary}

\subsection{Proofs of the main theorems}\label{ss: proof talk} 
Let us now discuss some aspects of the proofs of our main theorems. We begin by discussing Theorem~\ref{t:stability}, and particularly its connection to Theorem~\ref{thm: quantitative FK general} in the case of the round sphere.
The proof of Theorem \ref{t:stability} is accomplished by quantifying the following estimates when $J(1)-J(\rho)$ is small: 
 \begin{enumerate}
  \item For a large set of $r \in [\rho,1]$ and for $i =1,2$, the restriction of $u_i$ to $\pa B_r$  is close in an $L^2$ sense to the first Dirichlet eigenfunction of the spherical Laplacian on the set $\{u_i>0\}\cap \partial B_r$ . \\
  \item For $i=1,2,$ the function $u_i$ is close in an $L^2$ sense to a one-homogeneous function in $B_1\setminus B_{\rho}$. \\
  \item For a large set of $r \in [\rho,1]$ and for $i=1,2$, the first Dirichlet eigenvalue of the spherical Laplacian on the set $\{u_i>0\}\cap \partial B_r$ is close to the first eigenvalue of a  hemisphere in $\pa B_r$. 
 \end{enumerate} 
 
The one-homogeneous extension of the first eigenfunction of a hemisphere in $\pa B_1$ is a truncated linear function $(x \cdot \nu)^+$ for some $\nu \in S^{n-1}$.  Theorem~\ref{t:stability} is established by combining the estimates above  to show that for $i=1,2$, the function $u_i$ is close to a constant multiple $\beta_i>0$ of this linear function, for suitably chosen $\nu$. 
 
 In the case when $n=2$,  the proof of Theorem~\ref{t:stability} is significantly simpler than in higher dimensions, fundamentally because the geometry of $S^1(r) = \pa B_r$ is sufficiently simple that much of the analysis can be done in an explicit manner. In dimension $2$, the estimate (3) for the the first eigenvalue of $\{u_i>0\}\cap \partial B_r$ immediately determines the length of the longest connected curve of $\{u_i>0\}\cap \partial B_r$ as well  as the explicit first eigenfunction on $\{u_i>0\}\cap \partial B_r$.

  When $n>2$, on the other hand, the estimate (3) contains less immediate information. Knowing the first eigenvalue on $\{u_i>0\}\cap \partial B_r$ does not 
 determine the shape of $\{u_i>0\}\cap B_r$ nor the eigenfunction on $\{u_i>0\}\cap B_r$. It is here that we make crucial use of the sharp quantitive version of Sperner's inequality established in Theorem~\ref{thm: quantitative FK general}. Theorem~\ref{thm: quantitative FK general} applied to $S^{n-1} =\pa B_1$,  in conjunction with the estimate (3) allows us to deduce that the restriction of $u_i$ to $\pa B_1$ is quantitatively close to the first eigenfunction of a spherical cap of the same volume. We emphasize that for this application, it is essential that we have established the improved form of stability that includes eigenfunction control in Theorem~\ref{thm: quantitative FK general}.

\medskip

Let us now make some further remarks about Theorem~\ref{thm: quantitative FK general} and its proof. We begin with some comments about the statement and optimality of the theorem.

\begin{remark}[Dependence of constants on the volume]
	{\rm In Theorem~\ref{thm: quantitative FK general}, the constant $c$ in the quantitative estimates depend on the volume.  On Euclidean space, the scaling invariance can be used to state the Faber-Krahn inequality without a volume constraint as 
	$
\lambda_1(\Omega) |\Omega|^{2/n} \geq  \lambda_1(B) |B|^{2/n}
$ where $B=B_1$,
and similarly \eqref{eqn: stability statement} can be restated in scale-invariant form as 
\[
\lambda_1(\Omega) |\Omega|^{2/n} - \lambda_1(B) |B|^{2/n} \geq c \left( \inf_{x \in \R^{n}}\frac{| \Omega \Delta B(x,r)|^2}{|\Om|^2} + \inf_{x \in \R^n} \int_{\R^n} | u_\Omega - u_{B(x,r)}|^2 \right)\,
\]
where $c$ is a constant depending on dimension only and $r$ is chosen so that $|\Om| = |B_r|$.
On the sphere and hyperbolic space, the constant can clearly be taken to be uniform for any compact subset of volumes. We make no attempt to track the dependence of the constant $c=c(n,v)$ in \eqref{eqn: stability statement} as $v\to 0$ or as  $v \to |S^n|$ (in the case of the sphere) or $v\to \infty$ (in the case of hyperbolic space), though as $v\to 0$ one expects that the constants can be made uniform after the suitable rescaling. 
	}
\end{remark}

\begin{remark}[Optimality of the $L^2$ norm]
	{\rm It is natural to wonder if the $L^2$ distance between eigenfunctions  in \eqref{eqn: stability statement} can be strengthened to the $H^1$ distance. This is false, at least with the quadratic power. Indeed, let $\B$ be a ball of maximal overlap with $\Om$ and suppose by way of contradiction that there exists $c>0$ such that $\lambda_1(\Omega)-\lambda_1(\B) \geq c \int |\na u_\Omega - \na u_\B|^2 $. On $\B\setminus \Omega$, we have $\na  u_\Omega = 0$ and $|\na u_\B |\approx 1.$  So, 
	\[
	\int |\na u_\Omega - \na u_\B|^2 \geq \int_{\B\setminus \Omega} |\na u_\B|^2 \geq c\, |\B\setminus \Omega | = \frac{c}{2}\, |\B\Delta \Omega|.
	\]
This would imply that the eigenvalue deficit $\lambda_1(\Omega) -\lambda_1(\B)$ controls the asymmetry {\it linearly}, which as we have already discussed is false.	
It is not hard to show that {\it qualitative} stability holds for the $H^1$ norm and we do not know if quantitative stability may hold with a different modulus of continuity, e.g. for some $\alpha>2$ in \eqref{eqn: stability general form}.
	}
\end{remark}

\begin{remark}[Quotient spaces]
	{\rm 
	Naturally, the Polya-Szeg\"{o} principle and thus the Faber-Krahn inequality hold on quotients of Euclidean space, hyperbolic space, and the round sphere when  the volume of $\Omega$ is small. Theorem~\ref{thm: quantitative FK general} also generalizes to quotient spaces for the volume $v$ beneath a suitable threshold. 
	}
\end{remark}

%
%
%
%

The proof of  Theorem~\ref{thm: quantitative sperner}, like the proof of the quantitative Faber-Krahn inequality of \cite{BDV15}, is based on a {\it selection principle} argument. This global-to-local reduction tool, introduced by Cicalese and Leonardi in \cite{CiLe12} in the context of the isoperimetric inequality, is perhaps the most robust method to prove quantitative geometric inequalities and allows one to trade the task of proving stability in the entire class of objects for the task of proving stability among minimizers of some penalized functional, which typically enjoy strong regularity properties. 

Let us sketch the basic scheme of the selection principle argument. Suppose by way of contradiction that the desired stability result \eqref{eqn: stability statement} fails, and so there exist a sequence of open bounded sets $\Omega_j$ of a fixed volume $v$ with $\lambda_1(\Omega_j) \to \lambda_1(\B)$ and yet 
\begin{equation}\label{eqn: contra general}
\lambda_1(\Omega_j) - \lambda_1(\B)\leq \frac{\text{dist}(\Omega_j, \B)^2}{j},
\end{equation}
where we let $\text{dist}(\Omega,\B)$ denote the distance on the right-hand side of \eqref{eqn: stability statement}.
 Now, consider a functional of the form 
 \begin{equation}\label{eqn: functional}
 \mathcal{F}_j(\Om) = \lambda_1(\Omega) +\tau \sqrt{c_j^2 + (c_j - \text{dist}(\Omega, B)^2)^2},
 \end{equation}
  where $c_j =\text{dist}(\Omega_j, B)^2$ which is designed so that a minimizer $\Omega_j'$ of $\mathcal{F}_j$ among sets of a fixed volume  satisfies the same contradiction assumption \eqref{eqn: contra general}. 
 Theorem~\ref{thm: quantitative FK general} will then follow once we establish the following two main ingredients for a selection principle: 
\begin{enumerate}
	\item (Regularity) Minimizers $\Omega_j'$ of the penalized functionals $\mathcal{F}_j$ satisfy good $\e$-regularity properties. Combined with the (easily shown) fact that the $\Omega_j'$ converge to a ball $\B$ in a weak sense, this will guarantee that for $j$ sufficiently large, the boundary of $\Omega_j'$ is a small $C^{2,\alpha}$ perturbation of the boundary of $\B$.\\
	  
	\item (Local stability) The quantitative stability estimate \eqref{eqn: stability statement} holds for sets whose boundaries are small $C^{2,\alpha}$ perturbations of the boundary of $\B$. This step is based on {\it linear stability,} i.e. strict positivity of the second variation of $\lambda_1(\B)$ in non-translational directions.
\end{enumerate}

In the context of Theorem~\ref{thm: quantitative sperner}, significant challenges arise in both of these two steps. Let us describe these challenges, and
  the new strategies introduced to deal with them.

{\it Step (1):}  For an energy like \eqref{eqn: functional} in which $\lambda_1(\Om)$ (or the closely related torsional rigidity below) is the dominant term, the functional $\mathcal{F}_j$ can be realized  
as  a perturbation of the Alt-Caffarelli functional \cite{AC81}. 
However, in contrast to \cite{BDV15}, the presence of $L^2$ eigenvalue penalization in the distance makes this perturbation critical
and leads to significant issues for regularity. Step (1) is carried out in our companion paper \cite{AKN1}, where we consider a selection principle associated to the geometric inequality
\[
\lambda_1(\Omega) + \frak{T} \tor(\Omega) \geq \lambda_1(B) + \frak{T} \tor (B)
\]
for a small number $\frak{T}>0$. Here $\tor(\Omega)$ is the (negative of the physical quantity) {\it torsional rigidity} 
of an open set $\Omega \subset \R^n$, defined and discussed further in Section~\ref{sec: SP}.
 In Section~\ref{sec: SP}, we will see how to connect this selection principle with the one we need to prove Theorem~\ref{thm: quantitative FK general}, using the so-called Kohler-Jobin inequality.

{\it Step (2):} In nearly all  selection principle arguments, including  the quantitative Faber-Krahn inequality of \cite{BDV15} in Euclidean space, Step (2) can be carried out using an {\it explicit} analysis of the second variation. 
On the round sphere and hyperbolic space, these computations can no longer be carried out explicitly.  To overcome this challenge,  we introduce a new technique, based on the maximum principle, to perform an {\it implicit spectral analysis}.
We exploit the symmetry of the minimizers, which are balls, (thus indirectly using their explicit form), and use an ODE maximum principle argument to establish the desired spectral gap. 
 To our knowledge, this is the first time this method has been used in a quantitative stability estimate or in a spectral analysis otherwise.
\\

\medskip

The remainder of the paper is organized in the following way. In Section~\ref{sec: qacf} we prove the quantitative estimates for the Alt-Caffarelli-Friedman monotonicity formula of Theorem~\ref{t:stability}. We present the proof in the cases $n=2 $ and $n\geq 3$ separately because the two-dimensional case is simpler and provides a basic outline of the proof for higher dimensions. 
Corollary~\ref{c:uniqueblowup} is also proven in this section.

In Section~\ref{sec: sperner}, we prove Theorem~\ref{thm: quantitative FK general}. In the interest of clarity, we present the proof in the specific case of the round sphere, which is also the essential case in our application to Theorem~\ref{t:stability}. The proof carries over almost verbatim to Euclidean space and hyperbolic space, and we therefore only remark on the necessary modifications to the proof in these cases.  As we have mentioned, the regularity estimates of Step (1), as well as the details of the selection prinicple, are shown in our companion paper \cite{AKN1}. Section~\ref{sec: sperner} is thus largely dedicated to proving Step (2) in Section~\ref{sec: NSS}, with the regularity results of \cite{AKN1} recalled and the conclusion of Theorem~\ref{thm: quantitative FK general} drawn in Section~\ref{sec: SP}.


%
%
%
%
\section{Quantitative Stability for the Alt-Caffarelli-Friedman Monotonicity Formula}\label{sec: qacf}
This section is dedicated to the proof of Theorem~\ref{t:stability} and Corollary~\ref{c:uniqueblowup}. The cases $n=2$ and $n\geq 3$ are carried out separately in Sections~\ref{sec: 2d qacf} and \ref{sec: nd qacf} respectively. Corollary~\ref{c:uniqueblowup} is shown in Section~\ref{ssec: blowups}. We note that if $\log(J(1)/J(\rho)) \geq 1$, then Theorem \ref{t:stability} is true simply by choosing $\beta_i=\| u \|_{L^2}(B_1)$ and $\nu = e_1$ and choosing $C$ large enough depending only on $n$. Therefore, throughout this section we assume that $\log(J(1)/J(\rho)) <1$. 
We begin by showing that it is sufficient to prove Theorem \ref{t:stability} with the inequality 
\[
 \int_{B_1 \setminus B_{\r}} \left(u_1 -\beta_1 (x\cdot \nu)^+ \right)^2 + \left(u_2 -\beta_2 (x\cdot \nu)^-\right)^2
  \leq C\log\left(\frac{J(1)}{J(\r)}\right) \sum_{i=1}^2 \| u_i \|_{W^{1,2}(B_1)}^2. 
\]
Notice that in the above inequality the norm is $W^{1,2}$ rather than $L^2$. This next Lemma will show how  control the stronger $W^{1,2}$ norm with the weaker $L^2$ norm. 

\begin{lemma} \label{l:improved}
 Let $u_1, u_2$ satisfy the assumptions of Theorem \ref{t:stability}.  
 There exists a constant $C$ depending only on dimension $n$ such that 
 \[
  \int_{B_1} |\nabla u_i|^2  \leq C \frac{J(1)}{J(1/2)} \int_{B_1} u_i^2, 
 \]
 for $i=1,2$. 
\end{lemma}

\begin{proof}
 We note that multiplying either $u_i$ by a constant leaves $J(1)/J(1/2)$ invariant.  
 Thus, we may assume without loss of generality that 
 \[
  1=\int_{B_1} \frac{|\nabla u_1|^2}{|x|^{n-2}} = \int_{B_1} \frac{|\nabla u_2|^2}{|x|^{n-2}}. 
 \]
Then for any $\rho \in [1/2,3/4]$ we have 
 \[
 \begin{aligned}
   \int_{B_1} |\nabla u_2|^2 &\leq 
   \int_{B_1} \frac{|\nabla u_2|^2}{|x|^{n-2}} \\
  &=   \left(\int_{B_1} \frac{|\nabla u_1|^2}{|x|^{n-2}}\right) \left(\int_{B_1} \frac{|\nabla u_2|^2}{|x|^{n-2}}\right) \\
  &\leq 2^5\left( \frac{J(1)}{J(1/2)}\right)  \left(\int_{B_{1/2}} \frac{|\nabla u_1|^2}{|x|^{n-2}}\right) \left(\int_{B_{1/2}} \frac{|\nabla u_2|^2}{|x|^{n-2}}\right)  \\
  &\leq  2^5\left( \frac{J(1)}{J(1/2)}\right)  \left(\int_{B_{1}} \frac{|\nabla u_1|^2}{|x|^{n-2}}\right) \left(\int_{B_{\rho}} \frac{|\nabla u_2|^2}{|x|^{n-2}}\right)  \\
  &=2^5\left( \frac{J(1)}{J(1/2)}\right)  \left(\int_{B_{\rho}} \frac{|\nabla u_2|^2}{|x|^{n-2}}\right)  \\
  \end{aligned}
 \]
 
 We now use the fact that $u_2 \geq 0$, and $\Delta u_2 \geq 0$ to obtain (see \cite{CaffSalsa}) that 
 \[
  \int_{B_{\rho}} \frac{|\nabla u_2|^2}{|x|^{n-2}} \leq  C \int_{\partial B_{\rho}} |\nabla u_2|^2 + u_2^2.  
 \]
 Now there exists $\rho_1 \in [1/2,3/4]$ such that 
 \[
  \int_{\partial B_{\rho_1}} |\nabla u_2|^2 \leq 4 \int_{B_{3/4}} |\nabla u_2|^2. 
 \]  
 We now use that $u_2$ is subharmonic and the Caccioppoli inequality to conclude that 
  \[
  \int_{\partial B_{\rho_1}} |\nabla u_2|^2 \leq 4 \int_{B_{3/4}} |\nabla u_2|^2 \leq C \int_{B_1} u_2^2. 
 \]  
 Using that $u_2$ is subharmonic we have that for any $\rho > \rho_1$ that 
 \[
  \int_{\partial B_{\rho_1}} u_2^2 \leq \frac{\rho_1^{n-1}}{\rho^{n-1}} \int_{\partial B_{\rho}} u_2^2 \leq  \int_{\partial B_{\rho}} u_2^2. 
 \]
 Then
 \[
 \begin{aligned}
   \int_{\partial B_{\rho_1}} u_2^2 &\leq \frac{1}{1- 3/4} \int_{3/4}^1  \int_{\partial B_{\rho_1}} u_2^2 \ d\rho \leq 4 \int_{3/4}^1  \int_{\partial B_{\rho}} u_2^2 \ d\rho \\
   &= 4\int_{B_1 \setminus B_{3/4}} u_2^2   \leq 4\int_{B_1} u_2^2.  
  \end{aligned}
 \]
 Thus we have shown that 
 \[
  \int_{B_1} |\nabla u_2|^2 \leq C \left(\frac{J(1)}{J(1/2)}\right) \int_{B_1} u_2^2. 
  \]
  The same argument applies to $u_1$. 
\end{proof}

\subsection{The two dimensional case}\label{sec: 2d qacf}

Throughout the section,  we let $n=2$ and let $u_1,u_2:B_1 \to \R$ be two nonnegative functions satisfying \eqref{eqn: harmonic functions}. Let $J=J[u_1, u_2]:(0,1)\to \R_+$ be the two dimensional Alt-Caffarelli-Friedman monotonicity formula defined in \eqref{eqn: ACF 2d}.
 
To begin, we prove the monotonicity of the functional $J(r)$ following \cite{acf84, CaffSalsa}, along the way introducing some notation and outlining the main ideas of the proof of Theorem~\ref{t:stability} when $n=2$. 
We directly compute
\begin{equation}\label{eqn: derivative1 2d}
	\log(J(r))' = \frac{\int_{\pa B_r} |\na u_1|^2}{\int_{B_r} |\na u_1|^2}  +\frac{\int_{\pa B_r} |\na u_2|^2}{\int_{B_r} |\na u_2|^2} -\frac{4}{r}.
\end{equation}
 For each fixed $r \in (0,1]$ and for $i=1,2$, we denote by $\{\lambda_{k,r}^i\}$ and $\{Y_{k,r}^i\}$ the Dirichlet eigenvalues and eigenfunctions of the Laplacian $-\pa_{\tau \tau}$ on the set $\Omega_r^i= \{u_i>0\} \cap \pa B_r$, with the eigenfunctions normalized so that $\|Y_{k,r}^i\|_{L^2(\Omega^i_r)}=1$. Here and in the sequel, we let $\pa_\tau$ denote the tangential derivative on $\pa B_r,$ so that in polar coordinates $(r, \theta)$ for $\R^2$, we have  $\pa_\tau = \frac{1}{r} \pa_{\theta}$. The positivity set $\Omega_r^i$ comprises a union of circular arcs in $\pa B_r$, and so the the first eigenvalue and eigenfunction of this set are the (explicit) first eigenvalue and eigenfunction of the longest circular arc. More specifically, define  $\theta_i(r) \in [0, 2\pi]$  so that the  length of the longest  circular arc in $\Omega_r^i$  is  $r \theta_i(r)$. 
The first eigenvalue and eigenfunction of $\Omega_r^i$ are given by (up to the obvious rotation)
\begin{equation}\label{eqn: 2d first eval}
\lambda_{1,r}^i = \frac{\pi^2}{r^2 \theta_i(r)^2}, \qquad\qquad Y_{1,r}^i = \sqrt{\frac{2}{r \theta_i(r)}}\, \sin\Big(r \sqrt{\lambda_ {1,r}^i}\theta \Big)^+ .
\end{equation}
Here we have written the eigenfunction in polar coordinates $(r, \theta)$ for $\R^2$. It is worth noting that for a semicircle in $\pa B_r,$ the first eigenfunction is $\lambda = 1/r^2$ and the first eigenfunction is $Y= \sqrt{2/\pi r} \sin( \theta)^+,$ whose one-homogeneous extension is a linear function restricted to a half plane.
Observe that $\theta_1(r) + \theta_2(r) \leq 2\pi$ because the supports of $u_1$ and $u_2$ are disjoint.    

Now, we note that for $i=1,2$,
\[
\begin{aligned}
 &\int_{\partial B_r} |\nabla u_i|^2 
 = \int_{\partial B_r} (\partial_{\tau} u_i)^2 + (\partial_r u_i^2)
 \\
 & = \left( \int_{\partial B_r} (\partial_{\tau} u_i)^2 - \lambda_{1,r}^i u_i^2\right) 
  +\left(\int_{\partial B_r} \lambda_{1,r}^i u_i^2 - 2\sqrt{\lambda_{1,r}^i} u_i \partial_r u_i \right) + \left( \int_{\partial B_r} 2\sqrt{\lambda_{1,r}^i} u_i \partial_r u_i\right) . 
 \end{aligned}
\]
Using now the subharmonicity and nonnegativity of $u_i$ we have 
\[
 \int_{B_r} |\nabla u_i|^2 \leq \int_{\partial B_r} u_i  \pa_r u_i, 
\]
so that dividing by $\int_{\partial B_r} |\nabla u_i|^2$, summing in $i$, and subtracting $4/r$ we obtain from \eqref{eqn: derivative1 2d} and \eqref{eqn: 2d first eval} that 
\begin{equation}\label{eqn: logJprime and deltas}
	\log(J(r))'  \geq \delta_A(r) + \delta_B(r)  +\delta_C(r)\,,
\end{equation}
where we define
\begin{align*}
	\delta_A(r) &= \sum_{i=1}^2 \ \frac{1}{\int_{B_r}| \na u_i|^2}\left( \int_{\pa B_r} (\pa_\tau u_i)^2 -\lambda_{1,r}^i\int_{\pa B_r} u_i^2\right)\,,
	\\
	\delta_B(r) & =\sum_{i=1}^2 \ \frac{1}{\int_{B_r}| \na u_i|^2}
	\left(\lambda_{1,r}^i\int_{\pa B_r}  u_i^2 + \int_{\pa B_r}(\pa_r u_i)^2 -2\sqrt{\lambda_{1,r}^i} \int_{\pa B_r} u_i \pa_r u_i\right)\,,
	\\
	\delta_C(r) & = \frac{2}{r} \left( \frac{\pi}{\theta_1(r)} +  \frac{\pi}{\theta_2(r)}-2\right)\,.
\end{align*}
The ACF functional $J(r)$ is shown to be monotone in \cite{acf84, CaffSalsa} by showing the nonnegativity of each of $\delta_A(r), \delta_B(r)$, and $\delta_C(r)$; this is in particular a consequence of Lemma~\ref{l:3estimate} below. The quantitative monotonicity formula of  Theorem~\ref{t:stability} will be established by quantifying the loss in each inequality $\delta_{(\cdot)}(r) \geq 0$ and integrating with respect to $r$. Let us informally discuss the information contained in each $\delta_{(\cdot)}(r)$.

The quantity $\delta_A(r)$ is nonnegative by the Poincar\'e inequality on $\Omega_r^i$, and equals zero for a given $r \in (0,1]$ if and only if $u_i$ restricted to $\pa B_r$ is equal to the first eigenfunction $Y_{1,r}^i$ of $\Omega_r^i$. By quantifying the positivity of $\delta_A(r)$ in \eqref{eqn: deltaA est} in Lemma~\ref{l:3estimate} below, we see that $\delta_A(r)$ quantitatively controls the distance of $u_i$ to the eigenfunction on this set. 

The term in parentheses in the definition of $\delta_B(r)$ can be realized as a square as in \eqref{eqn: deltaB est} in Lemma~\ref{l:3estimate} below. This immediately shows that $\delta_B(r)$ is nonnegative for all $r \in (0,1)$. Furthermore, it provides quantitative control on how closely the restriction of $u_i$ to $\pa B_r$ behaves like a homogeneous function with homogeneity $\sqrt{\lambda_{1,r}^i}/r$. 

The quantity $\delta_C(r)$ can be bounded below by square functions; see \eqref{eqn: deltaC est} in Lemma~\ref{l:3estimate} below, which in particular shows that $\delta_C(r)\geq 0$. The estimate \eqref{eqn: deltaC est} shows that  the longest circular arc in $\Omega_r^i$ (and thus $\Omega_r^i$ itself) is quantitatively close to a semicircle in $\pa B_r$, and the corresponding first eigenvalue $\lambda_{1,r}^i$ is quantitatively close to the first eigenvalue $1/r^2$ of a semicircle in $\pa B_r$.

The proof of Theorem~\ref{t:stability} roughly goes as follows. By integrating $\delta_B(r)$ and $\delta_C(r)$ from $\r$ to $1$, we show in Lemma~\ref{l:onehomogen} that each $u_i$ is quantitatively close to the one-homogeneous function defined by $r u_i(1,\theta)$ in polar coordinates. Then, in Proposition~\ref{l:approximate2}, we apply the estimate \eqref{eqn: deltaA est} for $\delta_A(1)$ to deduce that $ u_i(1,\theta)$ is almost equal to a multiple of first eigenfunction $Y_{1,1}^i$ of $\Omega_1^i$, and applying the estimate \eqref{eqn: deltaC est} once again, we see that $Y_{1,1}^i$ is quantitatively close to the first eigenfunction of the semicircle. Finally, Theorem~\ref{t:stability} follows from Proposition~\ref{l:approximate2} once we guarantee that the semicircles corresponding to $u_1$ and $u_2$ do not have too much overlap.  \\

 The following lemma provides basic quantitative lower bounds for each of $\delta_A(r), \delta_B(r)$, and $\delta_C(r)$ discussed above. Since each of the right-hand sides of \eqref{eqn: deltaA est}-\eqref{eqn: deltaC est} are nonnegative, Lemmma~\ref{l:3estimate} and \eqref{eqn: logJprime and deltas} in particular imply the monotonicity of $J(r)$.

\begin{lemma}  \label{l:3estimate} For each $r \in (0,1]$ and $i=1,2$, let $Y_{1,r}^i$ be the eigenfunction defined in \eqref{eqn: 2d first eval} and let $\beta_{1,r}^i = \int_{\pa B_r} u_i Y_{1,r}^i$. 
	 We have
	\begin{align}
	\label{eqn: deltaA est}	\delta_A(r) & \geq \sum_{i=1}^2 \ \frac{(\lambda_{2,r}^i -\lambda_{1,r}^i ) }{\int_{B_r}| \na u_i|^2} \int_{\pa B_r} \left(u_i-\beta_{1,r}^iY_{1,r}^i\right)^2,\\
	\label{eqn: deltaB est}
		\delta_B(r) &= \sum_{i=1}^2 \ \frac{1}{\int_{B_r}| \na u_i|^2} \int_{\pa B_r} \left( \sqrt{\lambda_{1,r}^i} u_i - \pa_r u_i\right)^2, \\
	\label{eqn: deltaC est}	
		\delta_C(r) &\geq \sum_{i=1}^2\frac{2}{r} \frac{(\pi - \theta_i(r))^2}{\theta_i(r)(2\pi - \theta_i(r))}\geq \sum_{i=1}^2\frac{1}{\sqrt{\lambda_{1,r}^i}} \left( \sqrt{\lambda_{1,r}^i} - \frac{1}{r}\right)^2.
	\end{align}
\end{lemma}

\begin{proof} We fix $r \in (0,1]$, and for notational simplicity omit the subscripts $r$ in the remainder of the proof.  We first prove \eqref{eqn: deltaA est}. For $i=1,2$, let $\beta_k^i = \int_{\pa B_r} u_i Y_k^i$, so, expanding the term in parentheses of $\delta_A(r)$ in this basis, we have
\begin{align*}
\int_{\pa B_r} (\pa_\tau u_i)^2-\lambda_1^i \int_{\pa B_r} u_i^2 
&= \sum_{k=1}^\infty \lambda_k^i (\beta_k^i)^2- \lambda_1^i \sum_{k=1}^\infty (\beta_k^i)^2 \\
&\geq (\lambda_2^i -\lambda_1^i )\sum_{k=2}^\infty (\beta_k^i)^2 =  (\lambda_2^i -\lambda_1^i )  \int_{\pa B_r} (u_i-\beta_1^iY_1^i)^2.
\end{align*}
Dividing this expression by $\int_{B_r} |\na u_i|^2 $ and summing from $i=1$ to $2$ gives \eqref{eqn: deltaA est}. Next, \eqref{eqn: deltaB est} is immediate by writing the integrand as a square.
Finally, to prove \eqref{eqn: deltaC est}, recall that $\theta_2(r) \leq 2\pi - \theta_1(r) $, and so
\begin{align*}
\delta_C(r) 
 \geq \frac{2}{r} \left( \frac{\pi}{\theta_1(r)}  +\frac{ \pi}{2\pi - \theta_1(r)} - 2\right) = \frac{2}{r}\left(\frac{2\pi^2}{\theta_1(r)(2\pi - \theta_1(r))}- 2\right)= \frac{4}{r} \frac{(\pi - \theta_1(r))^2}{\theta_1(r)(2\pi - \theta_1(r))}.
 \end{align*}
The analogous expression holds with $\theta_2$ in place of $\theta_2$. Summing these expressions  and dividing by $2$ gives the first estimate in \eqref{eqn: deltaC est}. For the second, use \eqref{eqn: 2d first eval} to see
\[
\frac{4}{r} \frac{(\pi - \theta_1(r))^2}{\theta_1(r)(2\pi - \theta_1(r))}
 = \frac{1}{\sqrt{\lambda_1^1}}\frac{4\pi}{(2\pi - \theta_1(r) ) }\left( \sqrt{\lambda_1^1} - \frac{1}{r}\right)^2	\geq \frac{2}{\sqrt{\lambda_1^1}} \left( \sqrt{\lambda_1^1} - \frac{1}{r}\right)^2.
\]
Again, the same expression holds with $\theta_2$ in place of $\theta_2$. Summing these expressions  and dividing by $2$ gives the second estimate in \eqref{eqn: deltaC est}. 	This concludes the proof of the lemma.
\end{proof}

The following lemma makes use of the estimates for $\delta_B(r)$ and $\delta_C(r)$ in Lemma~\ref{l:3estimate} and shows that $u_1$ and $u_2$ are quantitatively close to being $1$-homogeneous.
\begin{lemma}   \label{l:onehomogen}
Let $u_1,u_2$ be as above. Then for a universal constant $C>0$ and any $\rho \in (0,1)$,
 \[
  \int_{B_1 \setminus B_{\rho}} [ru_i(1,\theta)-u_i(r,\theta)]^2 
   \leq C\log\left(\frac{J(1)}{J(\rho)}\right) \|u_i \|_{W^{1,2}(B_1)}^2 .  
 \]
\end{lemma}

\begin{proof}Fix $i \in \{1,2\}$. For notational simplicity, we omit the sub and superscripts $i$ in this proof. We first note that 
\[
\begin{aligned}
 \int_{B_1 \setminus B_{\rho}} [ru(1,\theta) - u(r,\theta)]^2 &\leq 2\int_{B_1 \setminus B_{\rho}} [ru(1,\theta) - A(r)u(1,\theta)]^2 \\
 &\quad+ 2\int_{B_1 \setminus B_{\rho}} [A(r)u(1,\theta) - u(r,\theta)]^2 = (I) + (II),
 \end{aligned} 
\]
where we define $A(r):=e^{-\int_r^1 \sqrt{\lambda(t)} dt}$ and $\lambda(t) = \lambda_{1,t}$. We bound $(II)$ in the following manner. 
Using that 
\[
u(1,\theta)A(r)-u(r,\theta)=\int_r^1 [u_r(s,\theta) - \sqrt{\lambda(s)}u(s,\theta)]e^{-\int_r^s \sqrt{\lambda(t)}dt} \ ds
\]
 as well as that 
$|e^{-\int_r^s \sqrt{\lambda(t)} dt }|\leq 1$ we have 
\begin{equation} \label{e:sq1}
 |u(1,\theta)A(r)-u(r,\theta)|^2 
 \leq \int_r^{1} |u_r(s,\theta)- \sqrt{\lambda(s)} u(s,\theta)|^2 \ ds.
\end{equation}
Now to continue with $(II)$ we have 
\[
\begin{aligned}
&\int_{B_1 \setminus B_{\rho}} [u(r,\theta) - A(r)u(1,\theta)]^2 
\\
&= \int_{\rho}^1 \int_{\partial B_1} r [u(r,\theta) - A(r)u(1,\theta)]^2 d\theta \ dr 
\\
& \leq \int_{\rho}^1 \int_{\partial B_1} r  \int_r^{1} |u_r(s,\theta)- \sqrt{\lambda(s)} u(s,\theta)|^2 \ ds  \  d \theta \ dr &\quad \text{  by } \eqref{e:sq1}
\\
& \leq \int_{\rho}^1 \int_{\partial B_1} \int_r^{1} s |u_r(s,\theta)- \sqrt{\lambda(s)} u(s,\theta)|^2 \ ds  \  d \theta \ dr  &\quad \text{ since } r \leq s\\
& \leq \int_{\rho}^1 \int_{\partial B_1} \int_{\rho}^{1} s |u_r(s,\theta)- \sqrt{\lambda(s)} u(s,\theta)|^2 \ ds  \  d \theta \ dr &\quad \text{ since } \rho \leq r\\
& \leq \int_{\rho}^1 \int_{\rho}^1  \delta_{B}(s) \left(\int_{B_s} |\nabla u|^2\right) \ ds   \ dr\\
& \leq \left(\int_{\rho}^1  \delta_{B}(s) \ ds \right) \int_{B_1} |\nabla u|^2.
\end{aligned}
\]

In order to bound $(I)$ we first note that if $\sqrt{\lambda(s)} \geq 2/s$, then from \eqref{eqn: deltaC est} we have that $\sqrt{\lambda(s)} \leq 4 \delta_C(s)$, so that 
$\max\{\sqrt{\lambda(s)}, 2/s\} \leq 4\delta_C(s)+ 2/s$. Now using that $e^{-x}$ is $1$-Lipschitz for $x \geq 0$ we have 
\[
 \begin{aligned}
  |A(r)-r| &= \left|e^{-\int_r^1 \sqrt{\lambda(s)} \ ds} -  e^{-\int_r^1 1/s \ ds}\right| \\
  &\leq \left|\int_r^1 \sqrt{\lambda(s)} - 1/s \ ds \right| \\
  &\leq \left(\int_r^1 \frac{(\sqrt{\lambda(s)} - 1/s)^2}{\sqrt{\lambda(s)}} \ ds \right)^{1/2} \left( \int_r^1 \sqrt{\lambda(s)} \ ds\right)^{1/2} \\
  &\leq \left(\int_r^1 \delta_C(s)\ ds \right)^{1/2} \left( \int_r^1 \max\{\sqrt{\lambda(s)},2/s \}\ ds\right)^{1/2} \\
  &\leq \left(\int_r^1 \delta_C(s)\ ds \right)^{1/2} \left( \int_r^1 4\delta_C(s) + 2/s \ ds\right)^{1/2} \\
  &\leq \left(\int_r^1 \delta_C(s)\ ds \right)^{1/2} \left( 4\log(J(1)/J(r))- 2\ln r \right)^{1/2} \\
  &\leq \left(\int_r^1 \delta_C(s)\ ds \right)^{1/2} \left( 4- 2\ln r \right)^{1/2},
 \end{aligned}
\]
so that 
\begin{equation} \label{e:sq2}
|A(r)-r|^2 \leq  \left(\int_r^1 \delta_C(s) \right)(4- 2 \ln r). 
\end{equation} 
To continue with $(I)$, we use the trace inequality on $B_1$ and the fact that $r\mapsto r(4-2\ln r)$ is bounded for $r \in (0,1]$ to find 
\[
 \begin{aligned}
  &\int_{B_1 \setminus B_{\rho}} [ru(1,\theta) - A(r)u(1,\theta)]^2 = \left(\int_{\partial B_1} u^2(1,\theta) \ d\theta \right)\int_{\rho}^1 r [r-A(r)]^2 \ dr \\ 
  &\leq C \| u \|_{W^{1,2}(B_1)}^2 \int_{\rho}^1 r  \left(\int_r^1 \delta_C(s)  \ ds \right)(4- 2 \ln r)  \ dr \leq C \| u \|_{W^{1,2}(B_1)}^2 \int_{\rho}^1   \left(\int_{\rho}^1 \delta_C(s) \ ds \right) \ dr 
  \\
  &\leq C \| u \|_{W^{1,2}(B_1)}^2  \int_{\rho}^1 \delta_C(s) \ ds.   \\
 \end{aligned}
\]
Recalling \eqref{eqn: logJprime and deltas} concludes the proof.
 \end{proof}

This next lemma is a straightforward computation for sine functions that will be used in the proofs of Proposition~\ref{l:approximate2} and Theorem~\ref{t:stability}. 
\begin{lemma} \label{l:shift}
There exists a universal constant $C>0$ such that for any $\hat{\theta} \in [0,2\pi)$ and any $\lambda>0$ with $|1-\sqrt{\lambda}|\leq 1/2$,	
	\begin{align}
\label{eqn: sin int 1}		\int_{S^1} [\sin(\theta + \hat{\theta} )^+ - \sin(\theta)^+]^2 &\leq C\hat{\theta}^2\\	
\label{eqn: sin int 2}	\int_{S^1} [\sin(\sqrt{\lambda} \theta)^+ - \sin(\theta)^+]^2 & \leq C(1-\sqrt{\lambda})^2.
	\end{align}
\end{lemma}

\begin{proof}
 To estimate \eqref{eqn: sin int 1}, from the bound $|\sin'(\theta)|\leq 1$, one has $(\sin(\theta+\hat{\theta})-\sin(\theta))^2 \leq C \theta^2$. If $\sin(\theta + \hat{\theta})>0$ and
 $\sin(\theta)\leq 0$, then $(\sin(\theta+\hat{\theta}))^2 \leq \hat{\theta}^2$. This is sufficient to conclude \eqref{eqn: sin int 1}. 
 
 To estimate \eqref{eqn: sin int 2} we break up the integral over $S^1$ into the regions $(0,\pi)\cup(\pi,\bar \theta)$ where we let $\bar \theta = \pi /\sqrt{\lambda}$. If $\sqrt{\lambda}>1$, then it suffices to only 
 consider the interval $(0,\pi)$. We first have
 \[
 \begin{aligned}
  \int_0^\pi [\sin(\sqrt{\lambda} \theta)^+ - \sin(\theta)^+ ]^2 &\leq  \int_0^\pi [\sin(\sqrt{\lambda} \theta) - \sin(\theta) ]^2 \\
  &= \pi - \frac{\sin((\sqrt{\lambda}-1)\pi)}{\sqrt{\lambda}-1} - \frac{\sin(2\sqrt{\lambda}\pi)}{4\sqrt{\lambda}} + \frac{\sin((\sqrt{\lambda}+1)\pi)}{\sqrt{\lambda}+1}
  \end{aligned}
 \]
 which is a smooth function in $\sqrt{\lambda}$ and has a minimum of zero at $\sqrt{\lambda}=1$, so there exists a universal constant $C$ such that 
 \[
  \int_0^\pi [\sin(\sqrt{\lambda} \theta)^+ - \sin(\theta)^+ ]^2 \leq C (1-\sqrt{\lambda})^2
 \]
 provided that $|\sqrt{\lambda}-1|\leq 1/2$. 
 Furthermore, using a change of variables and a Taylor expansion of $\sin^2(\cdot)$ at $\theta=0$, we have 
 \[
 \int_\pi^{\bar \theta} [\sin(\sqrt{\lambda} \theta)^+ - \sin(\theta)^+ ]^2=\int_\pi^{\bar \theta} \sin^2(\sqrt{\lambda} \theta) = \int_0^{\bar \theta-\pi} \sin^2 (\sqrt{\lambda} \theta)
  \leq C (1-\sqrt{\lambda})^3  
 \]
for a universal constant $C$ when $|\sqrt{\lambda} -1|\leq 1/2$. We conclude \eqref{eqn: sin int 2}. 
\end{proof}

The next proposition contains the core of the proof of Theorem~\ref{t:stability} and shows that $u_1$ and $u_2$ are each well-approximated by a linear function in a quantitative sense. From Lemma~\ref{l:onehomogen}, we know that $u_1$ and $u_2$ are each well-approximated by one-homogeneous functions. In this proposition, we show that these one-homogeneous functions can be well-approximated by linear functions.

\begin{proposition}\label{l:approximate2}   
 Assume $\log(J(1)/J(\r))\leq 1$ with $\rho \in [0,1/2]$.  For $i=1,2,$ let $\hat{\beta}_i = \beta_{1,1}^i = \int_{\pa B_1} u_i Y_{1,1}^i$ and let $\beta_i = \sqrt{2}\hat{\beta}_i/\sqrt{\pi}$. There exists $\theta_i \in [0,2\pi]$ such that if we let $v_i$ be the linear function defined in polar coordinates by $v_i=r\sin(\theta+\theta_i)$, then 
 \[
  \int_{B_1\setminus B_\r} [u_i-\beta_i v_i^+]^2 \leq C\log\left(\frac{J(1)}{J(\r)}\right) \| u_i \|_{W^{1,2}(B_1)}^2.
 \]
\end{proposition}

\begin{proof} 
Let 
\[
\e =\log\left(\frac{J(1)}{J(\r)}\right) = \int_\r^1 \log(J(r))'\,dr\,.
\] 
We note that it is sufficient to prove the result for $\e$ small. Since the integrand is nonnegative, it follows from Chebyshev's inequality that $\log(J(r))' \leq 10\e$ outside a set of measure $\e/10$. In particular, $\log(J(r))' \leq 10 \e$ for some $r \in (1-  \e/5, 1]$.
 So, by scaling we may assume that 
 \begin{equation}\label{eqn: scaling j'(1)}
 \log(J(1))'\leq \epsilon.
 \end{equation}
 
  Applying the estimate  \eqref{eqn: deltaC est} for $\delta_C(r)$ with $r=1$, we deduce that there exists a universal constant $C$ such that 
 \begin{equation}\label{eqn: useful}
 |\pi - \theta_i(1)| \leq C\sqrt{\epsilon} \quad \text{ and } \quad |\sqrt{\lambda_{1,1}^i} - 1| \leq C\sqrt{\epsilon}.  \quad \lambda_2 - \lambda\geq c.
 \end{equation}
Consequently, the other connected arcs have length less than $C \sqrt{\epsilon}$, and it follows that $\lambda_{2,1}^i=4\lambda_{1,1}^i$, so that 
\begin{equation}\label{eqn: useful2} 
 \lambda_{2,1}^i - \lambda_{1,1}^i \geq 1. 
\end{equation}

Fix $i \in \{1,2\}$. For the remainder of the proof, we omit subscripts $i$ for notational simplicity. After a rotation, we assume that the longest arc on which $u\cap \pa B_1$ is positive is $(0,\theta(1))$. After this rotation, we let $v = r\sin(\theta)$. 
 Here and in the remainder of the proof, we let $\lambda = \lambda_{1,1},$ i.e. $\lambda$ is the first Dirichlet eigenvalue of $\Omega_1$. Likewise, $\lambda_2$ denotes the second Dirichlet eigenvalue of $\Omega_1$.
We now apply the triangle inequality to split the integral on the left-hand side into three pieces as follows:
  \begin{equation}\label{eqn: starting point}
  	\begin{split}
    \int_{B_1\setminus B_\r} (u - \beta v^+)^2    & \leq 4\int_{B_1\setminus B_\r} (u - ru(1,\theta))^2  \\
  &  +4 \int_{B_1\setminus B_\r} (r u(1,\theta) -r\hat{\beta} Y_{1,1})^2 \\
 &   +4\int_{B_1\setminus B_\r}  (\hat{\beta} r  Y_{1,1}- \beta r\sin(\theta)^+)^2 \,.
  	\end{split}
  \end{equation}
We apply Lemma~\ref{l:onehomogen} to the first term on the right-hand side of \eqref{eqn: starting point}  to find that for a universal constant $C$, we have 
 \begin{equation}
\label{eqn: application of one hom lemma} 	
4\int_{B_1\setminus B_\r} (u - ru(1,\theta))^2 \leq C \e \| u\|_{W^{1,2}(B_1)}^2\,.
 \end{equation} 

Next, we use the estimate \eqref{eqn: deltaA est} for $\delta_A(r)$ with $r =1$ to bound the second term on the right-hand side of \eqref{eqn: starting point}. Indeed, by \eqref{eqn: deltaA est} and the the uniform eigenvalue gap \eqref{eqn: useful}, we see that   
\begin{equation}\begin{split}
\label{eqn: second term}
4 \int_{B_1\setminus B_\r} (r u(1,\theta) - \hat{\beta} rY_{1,1})^2 & \leq C \int_{\pa B_1} (u(1,\theta) - \hat{\beta} Y_{1,1})^2\\
	& \leq C \log(J(1))' \|\nabla u \|_{L^2(B_1)}^2\leq C\e \| \nabla u \|_{L^2(B_1)}^2\,
\end{split}	
\end{equation}
 for a universal constant $C$. In the final inequality we have made use of  \eqref{eqn: scaling j'(1)}.

Finally, the third term in \eqref{eqn: starting point} can be estimated using the eigenvalue estimate in \eqref{eqn: useful} and the estimate \eqref{eqn: sin int 2} from Lemma~\ref{l:shift}. 
 Indeed, noting that $\hat{\beta} \leq \| u\|_{L^2(\partial B_1)} \leq  C\| u\|_{W^{1,2}(B_1)}$, we conclude that
 \begin{equation}\label{eqn: sin bound}
 \begin{split}
 4\int_{B_1\setminus B_\r}  (\hat{\beta} r Y_{1,1} -& \beta r\sin(\theta)^+)^2  \leq 
\hat{\beta}^2  \int_{\pa B_1}\left[ Y_{1,1} - \sqrt{\frac{2}{\pi}} \sin(\theta)^+\right]^2  \\
& \leq C\hat{\beta}^2\int_{\pa B_1}\left[ \sqrt{\frac{2}{\theta(1)}} \sin(\sqrt{\lambda} \theta)^+  - \sqrt{\frac{2}{\pi}}\sin(\theta)^+\right]^2\\ 
& \leq C\hat{\beta}^2\int_{\pa B_1}\left[  \sin(\sqrt{\lambda} \theta)^+  -\sin(\theta)^+\right]^2 \\
&\quad+ C\hat{\beta}^2 \left(\frac{1}{\sqrt{\theta(1)}} - \frac{1}{\sqrt{\pi}} \right)^2 \int_{\pa B_1}\left[  \sin(\sqrt{\lambda} \theta)^+\right]^2 \\ 
& \leq C\hat{\beta}^2 (1-\sqrt{\lambda})^2+ C\hat{\beta}^2 (\pi -\theta(1))^2
\leq C \e \| u \|_{W^{1,2}( B_1)}^2.
\end{split} 
 \end{equation}
By combining \eqref{eqn: starting point}, \eqref{eqn: application of one hom lemma}, \eqref{eqn: second term} and \eqref{eqn: sin bound}, we conclude the proof of the proposition.
\end{proof}

We are now ready to prove Theorem~\ref{t:stability} in the case $n=2$.
\begin{proof}[Proof of Theorem~\ref{t:stability} for $n=2$]
As noted in the beginning of the section, if $\log(J(1)/J(\rho))\geq 1$, then we may choose $\beta_i=\| u \|_{L^2(B_1)}$ and $\nu = e_1$, and the result immediately follows for large enough $C$. If 
$\log(J(1)/J(\rho))<1$, then we may choose for $i=1,2$, the constants $\beta_1^i$ and $\theta_i$ be as in Proposition~\ref{l:approximate2}. In order to conclude Theorem~\ref{t:stability},  the only issue is that the functions $\sin(\theta + \theta_1)^+$ and 
 $\sin(\theta+ \theta_2)^+$ may not have disjoint support. However, from \eqref{eqn: useful} we point out that the overlap may not exceed $C\sqrt{\epsilon}$. Therefore, by  Lemma \ref{l:shift}, 
 we may rotate both functions to have disjoint support and thus obtain the result. 
\end{proof}
\medskip

\subsection{The $n$-dimensional case}\label{sec: nd qacf}
We now move to the proof of Theorem~\ref{t:stability} in higher dimensions. Throughout this section, let $n\geq 3$, let $u_1,u_2$ be as in \eqref{eqn: harmonic functions} and let $J=J[u_1, u_2] : (0,1] \to \R$ be the $n$-dimensional Alt-Caffarelli-Friedman monotonicity formula defined in \eqref{eqn: ACF higher d}.
 The proof of Theorem~\ref{t:stability} follows the same basic scheme as the proof in the two dimensional setting, wherein we estimate the positivity of $\log(J(r))'$ on each sphere $\pa B_r$. Whereas in the two dimensional case the positivity sets $\pa B_r \cap\{u_i>0\}$ were unions of circular arcs, in higher dimensions the geometry of the positivity sets may be more complicated. In particular, we can no longer hope to write the first eigenfunction and eigenvalue in an explicit way. The key tool to overcome this difficulty is Theorem~\ref{thm: quantitative sperner} (specifically utilized in \eqref{e3} in Proposition \ref{l:approx}). In this section only, we will use the notation $\BB$ to denote an $(n-1)$-dimensional spherical cap (i.e. geodesic ball in $S^{n-1}$), so as not to be confused with the notation $B$ used for a ball in $n$-dimensional Euclidean space.

To begin the $n$-dimensional setting, let us recall the proof from \cite{acf84, CaffSalsa} that the ACF formula \eqref{eqn: ACF higher d} is monotone. We directly compute
\begin{equation}\label{eqn: derivative1}
	\log(J(r))' = \frac{r^{2-n}\int_{\pa B_r} |\na u_1|^2}{\int_{B_r} \frac{|\na u_1|^2}{|x|^{n-2}}}  +\frac{r^{2-n}\int_{\pa B_r} |\na u_2|^2}{\int_{B_r} \frac{|\na u_2|^2}{|x|^{n-2}}} -\frac{4}{r}.
\end{equation}
For each $r \in (0,1)$ and $i=1,2$, set $\Om_i^r = \{ u_i>0 \} \cap \pa B_r$ . Let $\lambda_{1,r}^i$ and $\tilde{u}_{i,r}$ respectively denote the first Dirichlet eigenvalue and eigenfunction $-\Delta_\tau$ on $\Om_i^r$, where $\tilde{u}_{i,r}$  is normalized so that $\| \tilde{u}_{i,r}\|_{L^2(\Omega_i^r)}=1$. Here $\Delta_\tau$ denotes the Laplacian of $\pa B_r$.
 For $n>2$, we may no longer write $\lambda_{1,r}^i$ explicitly as we did in \eqref{eqn: 2d first eval} in the two dimensional case. For a hemisphere in $\pa B_r$, the first eigenvalue is $\lambda_1=(n-1)/r^2$ with corresponding eigenfunction given by the restriction a linear function on $\R^n$ to this set. (By symmmetry, one can see that the first eigenfunction of a hemisphere is the restriction of the first spherical harmonic that vanishes on the boundary of the hemisphere, which up to a rotation is the second spherical harmonic.)
 
  An important quantity in this setting is the characteristic constant $\alpha=\alpha_i(r) \in (0, \infty)$ of $\Omega_i^r$ defined by
 \begin{equation}\label{eqn: characteristic constant}
 \alpha_i(r)^2 + (n-2)\alpha_i(r) - r^2\lambda_{1,r}^i=0,
 \end{equation}
 see  \cite{friedland1976eigenvalue}.
The characteristic constant is defined in such a way so that the $\alpha =\alpha_i(r)$ homogeneous extension  $w(x) =r^\alpha \tilde{u}_{i,r}$ of  $\tilde{u}_{i,r}$ is harmonic in the cone generated by $\Om_i^r.$ Like the first Dirichlet eigenvalue, the characteristic constant  is monotone decreasing under set inclusion, and among sets of a fixed volume is minimized by a spherical cap. A hemisphere in $\pa B_r$ has characteristic constant $\alpha =1$.
 
Since $u_i$ is subharmonic and $u_i(0)=0$ we have \cite{CaffSalsa} that 
\begin{equation} \label{e:cs}
 \int_{B_r} \frac{|\na u_i|^2}{|x|^{n-2}} \leq r^{2-n}\int_{\partial B_r} u_i (u_i)_r + \frac{n-2}{2}\frac{u_i^2}{r}. 
\end{equation}
 From \eqref{eqn: characteristic constant}, we see that we can express $\lambda_{1,r}^i$ as $\lambda_{1,r}^i = r^{-2}\alpha_i(r)^2 +(n-2)r^{-2} \alpha_i(r)$.
With this in mind, we now write $|\na u_i|^2$ in the following way for any point, where for a point $x \in \pa B_r$, we let $\na_\tau$ denote the projection of the gradient onto the tangent space of $\pa B_r$ at $x$.
\[
\begin{aligned}
|\na u_i|^2 =|\na_\tau u_i|^2 + (\partial_r u_i)^2  &= |\na_\tau u_i|^2 - \lambda_1^i u_i^2 + \lambda_1^i u_i^2 + (\partial_r u_i)^2 \\
 &=  |\na_\tau u_i|^2 - \lambda_1^i u^2 + \frac{\alpha_i^2}{r^2} u_i^2 + 2\frac{\alpha_i}{r} \frac{n-2}{2r} u_i^2 + (\partial_r u_i)^2 \\
 &=  |\na_\tau u_i|^2 - \lambda_1^i u^2  \\
 &\quad + \frac{\alpha_i^2}{r^2} u_i^2  - 2\frac{\alpha_i u_i}{r} \partial_r u_i + (\partial_r u_i)^2  \\
 &\quad +2\frac{\alpha_i}{r} u_i \partial_r u_i +2 \frac{\alpha_i}{r^2} \frac{n-2}{2} u_i^2.   \\
\end{aligned}
\]
Using this expression in the numerator of \eqref{eqn: derivative1}, as well as the fact that $|x|^{n-2} \leq r^{n-2}$ on $B_r$ and \eqref{e:cs}, we see  that $\log(J(r))' \geq \delta_A(r) + \delta_B(r)+ \delta_C(r)$ where 
\begin{align*}
	\delta_A(r) &= \sum_{i=1}^2 \ \frac{1}{\int_{B_r}| \na u_i|^2}\left( \int_{\pa B_r} (\na_\tau u_i)^2 -\lambda_1^i\int_{\pa B_r} u_i^2\right)
	\\
	\delta_B(r) & =\sum_{i=1}^2 \ \frac{1}{\int_{B_r}| \na u_i|^2}
	\left( \int_{\pa B_r}  \frac{\alpha_i(r)^2}{r^2}u_i^2 + (\pa_r u_i)^2 -2\frac{\alpha_i(r)}{r} u_i \pa_r u_i\right)
	\\
	\delta_C(r) & = \frac{2}{r}\left(\alpha_1(r) + \alpha_2(r)-2 \right)\,.
\end{align*}

In \cite{acf84, CaffSalsa}, the ACF functional $J(r)$ is shown to be monotone by showing that each of $\delta_A(r), \delta_B(r)$, and $\delta_C(r)$ is nonnegative, which we will see is a consequence of Lemma~\ref{l:33estimate} below. As in the two dimensional case, the quantitative monotonicity formula of  Theorem~\ref{t:stability} will be established by quantifying the loss in each inequality $\delta_{(\cdot)}(r) \geq 0$ and integrating with respect to $r$. Let us discuss the information captured by each $\delta_{(\cdot)}(r)$.

The quantity $\delta_A(r)$ is the direct analogue of $\delta_A(r)$ in the two dimensional setting. It is nonnegative by the Poincar\'e inequality on $\Omega_i^r$, and equals zero for a given $r \in (0,1]$ if and only if $u_i$ restricted to $\pa B_r$ is a multiple of the first eigenfunction $\tilde{u}_{i,r}$ of $\Omega_i^r$. As in the case $n=2$, $\delta_A(r)$ quantitatively controls the distance of $u_i$ to the eigenfunction on this set; see \eqref{eqn: deltaA est n} in Lemma~\ref{l:33estimate} below. 

The term in parentheses in the definition of $\delta_B(r)$ can be realized as a square as in \eqref{eqn: delta b est n dim} in Lemma~\ref{l:33estimate} below, and so $\delta_B(r)$ is nonnegative for all $r \in (0,1)$. It quantitatively controls how closely the restriction of $u_i$ to $\pa B_r$ behaves like a homogeneous function with homogeneity $\alpha_i(r)$. 

The nonnegativity of $\delta_C(r)$ comes from the Friedland-Hayman inequality \cite{friedland1976eigenvalue}.
In \eqref{eqn: deltaC esta} of Lemma~\ref{l:33estimate} below, we show that $\delta_C(r)$ provides quantitative control over how close the characteristic constant $\alpha_i(r)$ of  $\Omega_i^r$ is to the characteristic constant $\alpha=1$ of a hemisphere in  $\pa B_r$.

Let us outline the ingredients of the proof of Theorem~\ref{t:stability} that make up the rest of this section. By integrating $\delta_B(r)$ and $\delta_C(r)$ from $\r$ to $1$, we show in Lemma~\ref{l:onehomogen n} that each $u_i$ is quantitatively close to the one-homogeneous function defined by $r u_i(1,\theta)$ in polar coordinates $(r ,\theta ) \in \R_+ \times S^{n-1}$ for $\R^n$.  Lemmas~\ref{l:fh} and Lemma~\ref{l:eigengap} provide  quantitative and qualitative information about the proximity of $\Omega_i^1$ to a spherical cap when $\delta_C(1)$ is assumed to be small. Lemma~\ref{lem: caps} compares the eigenfunctions of spherical caps of different radii. In Proposition~\ref{l:approx}, we apply the estimate \eqref{eqn: deltaA est n} for $\delta_A(1)$ to deduce that $ u_i(1,\theta)$ is almost equal to a multiple of first eigenfunction $\tilde{u}_{1,1}^i$ of $\Omega_i^1$, and applying Lemma~\ref{lem: caps} and Theorem~\ref{thm: quantitative sperner}, we see that $\tilde{u}_{1,1}^i$ is quantitatively close to the first eigenfunction of a hemisphere. Finally, Theorem~\ref{t:stability} follows from Proposition~\ref{l:approx} once we guarantee that the hemispheres corresponding to $u_1$ and $u_2$ do not have too much overlap.  \\

\smallskip
 The following lemma is the higher dimensional analogue of Lemma~\ref{l:3estimate} and contains the  quantitative lower bounds for each of $\delta_A(r), \delta_B(r)$, and $\delta_C(r)$ discussed above. Since each of the right-hand sides of \eqref{eqn: deltaA est n}-\eqref{eqn: deltaC esta} are nonnegative, Lemmma~\ref{l:33estimate} in particular implies the monotonicity of $J(r)$. 
\begin{lemma}  \label{l:33estimate}
	For each $r\in (0,1]$ and $i=1,2$, let $\tilde{u}_{i,r}$ be the eigenfunction of $\Omega_i^r$ as  above and let $\beta_{1,r}^i = \int_{\pa B_r } u_i \tilde{u}_{i,r}$. We have
	\begin{align}
\label{eqn: deltaA est n} 		\delta_A(r) & \geq \sum_{i=1}^2 \ \frac{(\lambda_{2,r}^i -\lambda_{1,r}^i ) }{\int_{B_r}| \na u_i|^2} \int_{\pa B_r} \left(u_i-\beta_{1,r}^i \tilde{u}_{i,r} \right)^2\\
\label{eqn: delta b est n dim}		\delta_B(r) & = \sum_{i=1}^2 \ \frac{1}{\int_{B_r}| \na u_i|^2} \int_{\pa B_r} \left( \frac{\alpha_i(r)}{r} u_i - \pa_r u_i\right)^2 \\
\label{eqn: deltaC esta}	\delta_C(r) &\geq 
\sum_{i=1}^2  c(n)\frac{|\alpha_i(r) -1|^2}{r  \max\{\alpha_i(r),2\}},
	\end{align}
	with $c(n)$ a dimensional constant.  
\end{lemma}
\begin{proof}
The proof of \eqref{eqn: deltaA est n} is exactly the same as that of \eqref{eqn: deltaA est} of Lemma~\ref{l:3estimate}, and as in \eqref{eqn: deltaB est}, we recognize the integrand in $\delta_B(r)$ above as a square to see \eqref{eqn: delta b est n dim}.

Observe that the inequality in \eqref{eqn: deltaC esta} is straightforward when $\delta_C(r)$ is large. Therefore, it suffices to consider the case when $\delta_C(r)$ is small. It also suffices to consider the case when $\Omega_1^r$ and $\Om_2^r$ are both spherical caps. Indeed, let $\bar\alpha_1$ and $\bar \alpha_2$ denote the characteristic constants of disjoint spherical caps of the same volume as   $\Omega_1^r$ and $\Om_2^r$ respectively. Then by Sperner's inequality (recall \eqref{eqn: characteristic constant}), the corresponding $\delta_C(r)$ can only decrease for by replacing $\Omega_1^r$ and $\Om_2^r$ by spherical caps, and also using the Friedland-Hayman inequality, we see
 \begin{equation}\label{eqn: linear est}
 \delta_C(r) =\frac{1}{r} \left( |\alpha_1-\bar \alpha_1| + |\alpha_2 - \bar \alpha_2| + \bar \alpha_1 + \bar \alpha_2 -2\right) \geq \frac{1}{r}\left( |\alpha_1-\bar \alpha_1| + |\alpha_2 - \bar \alpha_2|\right).
 \end{equation}
That is, the deficit $\delta_C(r)$ controls $ |\alpha_1-\bar \alpha_1|$ {\it linearly}. Moreover, because the characteristic constant is monotone with respect to set inclusion, we may assume without loss of generality that  we have two complementary caps.

So, define the function $\hat{\alpha}:(0, 1)\to \R$  by letting $\hat{\alpha}(t)$ be the characteristic constant of the spherical cap of radius $2\pi t$. Let $\hat{\delta}(h)= \hat{\alpha}(1/2 +h) + \hat{\alpha}(1/2 -h) -2$, so that by the Friedland-Hayman inequality we have $\hat{\delta}(h) \geq 0$. 

We claim that  $\hat{\delta}(h)\geq c h^2$ for a universal constant $c$.
It is known \cite{CaffSalsa, friedland1976eigenvalue, KochNYMJ} that $\hat{\alpha}$ is convex with $\hat{\alpha}''(1/2) = 2c>0$, where $c$ is a universal constant, and therefore there exists $h_0$ such that  $\hat{\alpha}''(1/2 +h )\geq c$ for all $h$ with $|h|\leq h_0$, so that $\hat{\alpha}$ is \textit{strictly} convex in a neighborhood of $1/2$.  
  For the moment, we assume that $|h|\leq h_0$. 
So, a Taylor expansion shows that 
\begin{align*}
	\hat{\alpha}(1/2+ h) &\geq \hat{\alpha}(1/2)  + \hat{\alpha}'(1/2) h + \frac{c}{2}h^2 
\end{align*}
for any $h$ with $|h|\leq h_0$.
Replacing $h$ by $-h$ and summing the corresponding terms, we conclude that $\hat{\delta}(h)\geq c h^2$ for all $|h|\leq h_0$. By the convexity of $\hat{\alpha}$, it follows that 
$\hat{\delta}(h_1)\geq \hat{\delta}(h_0)$ for any $h_1 \geq h_0$. Therefore, by assuming that $\delta_C$ is small enough, we also have necessarily that $|h| \leq |h_0|$. 

Finally, since $\hat{\alpha} $ is a locally Lipschitz function, we have $|\hat{\alpha}(1/2 \pm h)-1|\leq C|h|$ from which it follows $|\overline{\alpha}_i(r) -1| \leq h$. Then 
\[
\begin{aligned}
|\alpha_i(r) -1|^2 &\leq 2|\alpha_i(r) - \overline{\alpha}_i(r)|^2 + 2|\overline{\alpha}_i -1|^2  \\
 &\leq 2(\alpha_1+\alpha_2-2)^2 + C(\alpha_1+\alpha_2-2) \\
 &\leq C(\alpha_1+\alpha_2-2),
\end{aligned}
\] 
with the last inequality following since $\alpha_1+\alpha_2-2$ is assumed to be small. By dividing by $r$ we  arrive at  \eqref{eqn: deltaC esta}. This concludes the proof.
\end{proof}  

The following lemma is the higher dimensional analogue of Lemma~\ref{l:onehomogen}. Using the estimates for $\delta_B(r)$ and $\delta_C(r)$ in Lemma~\ref{l:33estimate}, this lemma shows that $u_1$ and $u_2$ are quantitatively close to being $1$-homogeneous.
\begin{lemma}   \label{l:onehomogen n}
Let $u_1,u_2$ be as above. Then
 \[
  \int_{B_1 \setminus B_{\rho}} [ru_i(1,\theta)-u_i(r,\theta)]^2 
   \leq C\log\left(\frac{J(1)}{J(\rho)}\right) \| u \|_{W^{1,2}(B_1)}^2. 
 \]
\end{lemma}

\begin{proof}Fix $i \in \{1,2\}$. For notational simplicity, we omit the sub and superscripts $i$ in this proof.
The proof follows the proof of Lemma \ref{l:onehomogen} using Lemma \ref{l:33estimate} but with $A(r)$ now defined as $ A(r):=e^{-\int_r^1 \frac{\alpha(t)}{t} \ dt}.$
\end{proof}

In Lemma~\ref{l:33estimate}, we established a quantitative lower bound for $\delta_C(r)$ in term of the characteristic constants of the regions $\Omega^r_1$ and $\Omega_2^r$. In the following lemma, we prove some further quantitative estimates for $\delta_C(1)$ under the assumption that $\delta_C(1)$ is bounded above by one, which will be applied in the proof of Proposition~\ref{l:approximate2}.   Recall that we let $\BB$ denote an $(n-1)$-dimensional spherical cap. 
\begin{lemma}  \label{l:fh}
Assume that $\delta_C(1) \leq 1$ and fix $i=1,2$. Let $\bar  \BB_i \subset \pa B_1$ be a spherical cap such that $|\bar \BB_i |  = |\Omega_i^1|$, and let $\bar \lambda^i$ denote the first eigenvalue of $\bar \BB_i$. 
 There exists a constant $C(n)$  such that
 \begin{align}\label{eqn: char const  quant}
  \delta_C(1) &\geq  
  C\left( \left|\lambda^i_{1,1} - \bar \lambda^{i}\right| +\left| \bar \lambda^i - (n-1)\right|^2\right)\,.
 \end{align}
\end{lemma}
\begin{proof} We fix $i$ and suppress the dependence on $i$ and $r=1$ in the remainder of the proof.
Let $\bar \alpha$  denote the characteristic constant of $\bar \BB$. Note that if $\delta_C(1) \leq 1$ then by definition, $\alpha, \bar \alpha \leq 3$. We already showed in  \eqref{eqn: linear est} in the proof of Lemma~\ref{l:33estimate} that $|\alpha - \bar \alpha|$ is linearly controlled by $\delta_C(1)$. This  combined with 
  \eqref{eqn: characteristic constant} shows the first estimate in \eqref{eqn: char const  quant}:
 \begin{equation}\label{eqn: linear est eval}
 	 |\lambda - \bar \lambda|  \leq (n-2) |\alpha-\bar \alpha| + |\alpha^2 -\bar \alpha^2| \leq C  |\alpha-\bar \alpha| \leq C\delta_C(1).
 \end{equation}
Similarly, it follows from the proof of Lemma~\ref{l:33estimate} that $|\bar\alpha-1|$ is quadradically controlled by $\delta_C(1)$, and using  \eqref{eqn: characteristic constant} in the same way, we conclude the second estimate in \eqref{eqn: char const  quant} holds.
\end{proof}

The next lemma shows that for $\delta_C(1)$ sufficiently small, the sets $\Omega_1^1$ and $\Omega_2^1$ enjoy a uniform spectral gap.
\begin{lemma}   \label{l:eigengap} Fix $\eta>0$. There exists $\e_0 = \e_0(\eta, n)$ such that if $\delta_C(1) \leq \e_0$, 
 then $\lambda_2(\Omega_i^1) - \lambda_1(\Omega_i^1) \geq \eta$ for $i=1,2$. 
\end{lemma}
\begin{proof} We argue by contradiction. Suppose we may find a sequence of functions $(u_{1,k}, u_{2,k})$ satisfying \eqref{eqn: harmonic functions} such that $\delta_{C,k}(1) \to 0$ and $\Omega_k :=\{ u_{1,k}>0\} \cap \pa B_1$ has $|\lambda_2(\Omega_k) - \lambda_1(\Om_k)|\to 0$. 

By Lemma~\ref{l:fh}, we see that $|\lambda_1(\Omega_k) - (n-1)|\to 0$ and $|\lambda_1(\BB_k)- (n-1)|\to 0$, where $\BB_k$ is a spherical cap with $|\BB_k|=|\Omega_k|$. The first fact shows that $|\lambda_1(\Omega_k) -(n-1)|\to 0$, and the second ensures that $|\Omega_k| \to |\BB_+|$, where $\BB_+ \subset S^{n-1} $ denotes a hemisphere.

 Now, for each $k,$ let $\tilde{u}_{1,k}$ and $\tilde{u}_{2,k}$ respectively denote first and second eigenfunctions of $\Omega_k$, normalized to have $L^2$ norm equal to one and extended by zero to be defined on the entire sphere. So, there exist $\tilde{u}_{1, \infty}, \tilde{u}_{2, \infty} \in W^{1,2}(S^{n-1})$ such that  after passing to a subsequence,
 \begin{align*}
 	\tilde{u}_{1, k} \to \tilde{u}_{1, \infty} & 
 	 \ \text{ weakly in }W^{1,2}(S^{n-1}), \   \text{strongly  in }L^2(S^{n-1}), \ \text{pointwise a.e.}, \\
 	\tilde{u}_{2, k} \to  \tilde{u}_{2, \infty}& 
 	 \ \text{ weakly in }W^{1,2}(S^{n-1}),  \ \text{strongly  in }L^2(S^{n-1}), \ \text{pointwise a.e.}, 
 	 \end{align*}
and $ \int \tilde{u}_{1 , \infty} \cdot \tilde{u}_{2,\infty} = 0$.

 Let $\Omega = \{ \tilde{u}_{1, \infty} >0\}$. By lower semicontinuity of the norm, $\lambda_1(\Omega)\leq n-1$.	By Fatou's lemma, $|\Omega|  \leq \liminf |\{ \tilde{u}_{1,k} >0\}| = \lim |\Om_k| =|\BB_+|$. 
	From Sperner's inequality and the monotonicity of $\lambda_1$ with respect to  set inclusion, we see that up to a translation, $\Omega = \BB_+$ and $\tilde{u}_{1,\infty}$ is the first eigenfunction of $\BB_+$.
	
 Now, let $\Omega_2 = \{ \tilde{u}_{2, \infty} >0\}$. The same argument and the uniqueness of the first eigenfunction show that $\tilde{u}_{1,\infty} = \tilde{u}_{2,\infty}$, contradicting their orthogonality.
  This completes the proof. \end{proof}

\begin{lemma}\label{lem: caps} Fix a compact interval $[r_1,r_2]  \subset (0,\pi)$, let $r_1 \leq r,\bar r \leq r_2$,  and let $\BB,\bar \BB\subset \pa B_1$ denote spherical caps of radius $r, \bar r$ centered at the north pole. Let $\lambda$ and $\bar \lambda$ denote the first eigenvalue of $\BB$ and $\bar \BB$ respectively, and let $v$ and $\bar v$ denote the corresponding first eigenfunctions extended by zero to be defined on the whole sphere and normalized so that $\|  v\|_{L^2(S^{n-1})} = \| \bar v\|_{L^2(S^{n-1})} =1 $. There exist constants $C= C(r_1,r_2)$ and $\bar{c} = \bar{c}(r, \bar r)$ such that
\[
\int_{S^{n-1}} |\bar{c}v- \bar v|^2 \leq C(\lambda - \bar \lambda)^2.
\]	
\end{lemma}
\begin{proof}
	The functions $v$ and $\bar v$ are radially symmetric, and so with a slight abuse of notation we let $v$, $\bar v$  denote the radial part of the functions. These functions can be extended to $(r, \pi)$, $(\bar{r}, \pi)$ to satisfy
	\begin{align*}
	v''(\phi) +(n-2)\frac{\cos\phi}{\sin \phi} v'(\phi)& = - \lambda v(\phi) \text{ for }\phi \in (0, \pi),\\
	\bar v''(\phi) +(n-2)\frac{\cos\phi}{\sin \phi} \bar v'(\phi)& = - 
\bar \lambda \bar v(\phi) \text{ for }\phi \in (0, \pi),\\
	\end{align*}
with $v(r) = \bar v(\bar r) = 0$, $v'(0) = \bar v'(0) = 0$ and $v,\overline{v} \geq 0$ on $(0,r)$ and $(0,\overline{r})$ respectively.  
Setting $f_\lambda(t) =\bar c \,v(\cos^{-1}(t))$, where $\bar c $ is a constant to be chosen later in the proof, we see that $f_\lambda$ satisfies 
\[
 (1-t^2)f_\lambda''(t)-(n-1)tf_\lambda'(t)+\lambda f_\lambda(t)=0, 
\]
and likewise for $f_{\bar{\lambda}}(t) = \bar v (\cos^{-1}(t))$.
The solutions to the above equation are Legendre functions and are well known. Expanding $\displaystyle v=\sum_{k=0}^{\infty} a_k(\lambda) (1-t)^k$ as a series solution about
$t=1$ we obtain the recursion formula 
\begin{equation}\label{recursion}
 a_{k+1}(\lambda)= a_k(\lambda)\frac{k^2+(n-2)k-\lambda}{(k+1)(2k+n-1)}. 
\end{equation}
The series converges absolutely on  the interval $(-1,3)$. We normalize by choosing the zero order coefficient $\bar a_0$ for $ f_{\bar\lambda}$ to satisfy the constraint $\|\bar v\|_{L^2(S^{n-1})}=1$ and by choosing the zero order coefficient $a_0$ of $f_\lambda$ so that $a_0 = \bar a_0$, in this way determining the constant $\bar c$ in the statement of the lemma.  Then 
\begin{equation}\label{eqn: ak}
 a_{k+1}(\lambda) = \prod_{j=0}^k \frac{j^2 + (n-2)j - \lambda}{(j+1)(2j+n-1)}. 
\end{equation}
Taking the derivative of $\log a_{k+1}$ with respect to $\lambda$ we obtain 
\[
 |a_{k+1}'(\lambda)| =  |a_{k+1}(\lambda)| \sum_{j=0}^k \frac{1}{(j+1)(2j+n-1)} \leq M |a_{k+1}(\lambda)|. 
\]
Furthermore, from \eqref{recursion}, we see that for any compact interval $I \in (0 ,\infty)$, there is a constant $C=C(I)$ such that $ |a_{k+1}(\lambda)| \leq C/2^k$ for all $\lambda \in I$. 
Therefore, 
$|a_{k+1}(\lambda) - a_{k+1}(\overline{\lambda})| \leq C |\lambda - \overline{\lambda}|/2^k$ for any $\lambda, \bar\lambda \in I$. Then 
summing over $k$ we obtain for any $t \in [0,t_0]$ with $t_0 < \pi$ that 
\[
 |f_{\lambda}(t) - f_{\overline\lambda}(t)| \leq |\lambda - \overline{\lambda}| \sum_{k=0}^{\infty}\frac{C}{2^k} < C |\lambda - \overline{\lambda}|.
\]
 The conclusion then follows after changing variables back to $\phi$ and integrating over $\pa B_1$. 
\end{proof}

The next proposition is the higher dimensional analogue of Proposition~\ref{l:approximate2}. It shows that each of the $u_i$, $i=1,2$, is well approximated in $L^2$ by a linear function. 
\begin{proposition}
	 \label{l:approx} There exist $\e_0 = \e_0 (n)$ and $C=C(n)$ such that the following holds.
Assume $\log\left(\frac{J(1)}{J(\rho)}\right)\leq \e_0$ and $0 \leq \rho \leq 1/2$.
Let $v$ be the first eigenfunction of the spherical cap of radius $\pi/2$ in $\pa B_1$, extended by $0$ to be defined on $\pa B_1$. For $i=1,2$ there exist $\beta_i = \beta_i(u_i)$   such that, up to a rotation, we have 
\[
  \int_{B_1 \setminus B_{\rho}} \left(r\beta_i v - u_i \right)^2 \leq C \log\left(\frac{J(1)}{J(\rho)}\right) \| u_i \|_{W^{1,2}(B_1)}^2.
\]
\end{proposition}

\begin{proof} Let $\e =\log(J(1)/J(\rho))$.
Since $\e = \int_\rho^1 \log(J(r))'\,dr$ and the integrand is positive, it follows that $\log(J(r))' \leq 10\e$ outside a set of measure $\e/10$. In particular, $\log(J(r))' \leq 10 \e$ for some $r \in (1-  \e/5, 1]$.
 By scaling we assume that $\log(J(r))'\leq \e$ at $r=1$.

Fix $i \in \{1,2\}$. For the remainder of the proof, we omit subscripts $i$ for notational simplicity. Let $\Om = \pa B_1 \cap \{u>0\}$ and let $\tilde{u}$ denote the first eigenfunction of $\Omega$, normalized so that $\| \tilde{u}\|_{L^2(\Omega)}=1$ and extended by zero to be defined everywhere on 
$\pa B_1.$  Let $\bar \BB \subset \pa B_1$ denote a spherical cap centered at the north pole such that $|\bar \BB| = |\Omega|$, and let $\bar v$ denote its first eigenfunction with the normalization $\| \bar v\|_{L^2(\bar \BB)} =1$ and extended by zero to be defined on the whole $\pa B_1.$  Up to a rotation, we may assume that the infimum over $x\in S^{n-1}$ in \eqref{eqn: stability statement}  of Theorem~\ref{thm: quantitative sperner} is achieved at the north pole.
Lemmas~\ref{l:fh} and \ref{l:eigengap} guarantee that for dimensional constants $c$ and $C$,
 \begin{equation}\label{eqn: useful n}
 \left|\lambda_1 - \lambda(\bar\BB)\right| \leq C\e, \qquad \left|\lambda_1(\bar \BB) - (n-1)\right| \leq C\sqrt{\e}, \qquad \lambda_2 - \lambda_1\geq c, 
 \end{equation}

 Let $\hat{\beta} = \beta_1 = \int_{\pa B_1} u \tilde{u}$ be as in \eqref{eqn: deltaA est n}, and let $\bar c$ be as in Lemma~\ref{lem: caps} applied to  $r =  \pi/2$ and the radius $\bar r $ of $\bar \BB$. Let $\beta = \bar c \hat{\beta}$. We use the triangle inequality to see that
  \begin{equation}\label{eqn: starting point n}
  	\begin{split}
    \int_{B_1\setminus B_\rho} (u - \beta r v)^2  & \leq 4\int_{B_1\setminus B_\rho} (u - ru(1,\theta))^2 
   +4\int_{B_1\setminus B_\rho} (ru(1,\theta) - \hat{\beta} r\tilde{u})^2\\
  & + 4\hat{\beta}^2\int_{B_1\setminus B_\rho}  (r\tilde{u} - r\bar v)^2 +  4 \hat{\beta}^2\int_{B_1\setminus B_\rho}  (  r\bar v- r\bar{c} v)^2.
  	\end{split}
  \end{equation}

We can estimate the first two terms on the right-hand side of \eqref{eqn: starting point n} precisely the way we estimated their two-dimensional analogues in the proof of Proposition~\ref{l:approximate2}. Indeed, by Lemma~\ref{l:onehomogen n}, for the first term we have
\begin{equation}\label{e1}
	4\int_{B_1\setminus B_\rho} (u - ru(1,\theta))^2  \leq C \e \| u\|_{W^{1,2}(B_1)}^2.
\end{equation}
For the second term on the right-hand side of \eqref{eqn: starting point n}, we use the estimate \eqref{eqn: deltaA est n} for $\delta_A(1)\leq \e$  and the
 uniform eigenvalue gap of \eqref{eqn: useful n} to see that  for a dimensional constant $C$ we have
 \begin{equation}\label{e2}
 	 	4\int_{B_1\setminus B_\rho} (ru(1,\theta) - \hat{\beta} r\tilde{u})^2  \leq C \int_{\pa B_1} (u(1,\theta) - \hat{\beta} \tilde{u})^2 \leq  C\e\| u \|_{W^{1,2}(B_1)}^2 \,.
 \end{equation}

 To estimate the third term on the right-hand side of \eqref{eqn: starting point n}, we call upon the quantitative form of Sperner's inequality established in Theorem~\ref{thm: quantitative sperner}:
 \begin{equation}\label{e3}
 \begin{split}
 4\int_{B_1\setminus B_\rho} r^2 (\tilde{u} - \bar v)^2 &  \leq C \hat{\beta}^2 \int_{\pa B_1} (\tilde{u} - \bar v)^2\leq C\hat{\beta}^2 \left(\lambda_1 -\lambda_1(\bar \BB)\right) \leq  C\e\| u\|_{W^{1,2}(B_1)}^2 ,	
 \end{split}	
 \end{equation}
where in the final inequality we recall \eqref{eqn: useful n} and  that $\hat{\beta} \leq  \| u\|_{L^2(\pa B_1)} \leq C\| u\|_{W^{1,2}(B_1)}$.

 Finally, we estimate the fourth term in \eqref{eqn: starting point n}. Applying Lemma~\ref{lem: caps} with $r=\pi /2$ and $\bar r$  the radius of $\bar\BB$ followed by \eqref{eqn: useful n}, we see that 
 \begin{equation}\label{e4}
 	\begin{split}
 	4\hat{\beta}^2 \int_{B_1\setminus B_\rho}  (  r\bar v- r\bar v)^2 &\leq  C \hat{\beta}^2\int_{\pa B_1}  (\bar v - c v)^2 \\
	 &\leq C\hat{\beta}^2\left(\lambda_1(\bar \BB) - (n-1)\right)^2\leq C\e\|u\|_{W^{1,2}(B_1)}^2.		
 	\end{split}
 \end{equation}
 Combining \eqref{e1}--\eqref{e4}, we conclude the proof of the proposition.
\end{proof}

We now  prove Theorem~\ref{t:stability} in the case $n\geq 3$.

\begin{proof}[Proof of Theorem~\ref{t:stability} for $n\geq 3$]
Let $\epsilon_0$ be as in Proposition \ref{l:approx}. If $\log(J(1) / J(\rho)> \epsilon_0$, then we simply choose $\beta_i=\|u_i \|_{L^2(B_1)}$ and $\nu=e_1$, and the result follows trivially for large enough $C$. Let now $\e = \log(J(1) / J(\rho)$ and assume $\epsilon \leq \epsilon_0$. As in the proof of the two-dimensional case, Theorem~\ref{t:stability} almost follows from Proposition~\ref{l:approx}, but it remains to show that the truncated linear functions obtained for each of $u_1,u_2$ in Proposition~\ref{l:approx} can be taken to have disjoint supports. We also repeat the following notation from Proposition~\ref{l:approx}, now emphasizing the dependence on $i=1,2$:  let $\Omega_i = \{ u_i> 0\} \cap \pa B_1,$ let  $\bar \BB_i \subset \pa B_1$ be the spherical cap of the same volume as $\Omega^i$ with center achieving the infimum over $x \in S^{n-1}$ in \eqref{eqn: stability statement} of Theorem~\ref{thm: quantitative sperner},  and $\BB_i$ the hemisphere with the same center. We first show that 
\begin{equation}\label{eqn: sm diff}
| \BB_1 \Delta \BB_2| \leq C\e^{1/2}.
\end{equation}
Indeed, we repeatedly apply the triangle inequality to find
\[
| \BB_1 \Delta \BB_2| \leq \left(\left|\BB_1 \Delta \bar \BB_1\right| + \left|\BB_2 \Delta \bar \BB_2\right|\right) + \left( \left|\bar \BB_1 \Delta \Omega_1\right| + \left|\bar \BB_2 \Delta \Omega_2\right| \right) + \left| \Omega_1 \Delta \Omega_2\right|.
\]
Since the supports of $u_1$ and $u_2$ are disjoint, $| \Omega_1 \Delta \Omega_2| =0$. Next, we can apply Theorem~\ref{thm: quantitative sperner} to  $\Omega_1$ and $\Omega_2$ and use \eqref{eqn: useful n} to see that 
\[
\left|\bar{\BB}_1 \Delta \Omega_1\right| + \left|\bar {\BB}_2 \Delta \Omega_2\right| \leq C \left( \lambda_1\left(\Omega_1 \right) - \lambda_1(\bar \BB_2) \right)^{1/2} + \left(\lambda_1(\Omega_2) - \lambda_1 (\bar \BB_2)\right)^{1/2} \leq C\e^{1/2}.
\]
Next, using  \eqref{eqn: useful n}, we see that 
\[
\begin{aligned}
\left|\BB_1 \Delta \bar \BB_1\right| + \left|\BB_2 \Delta \bar \BB_2\right| &\leq C\left( \left| r_1 -\frac{\pi}{2}\right| - \left|r_2 - \frac{\pi}{2}\right|\right) \\
&\leq C\left(\left| \lambda_1(\bar \BB_1) - (n-1) \right| + \left|\lambda_1(\bar{\BB}_2) - (n-1)\right| \right)\leq C\e^{1/2}.
\end{aligned}
\]
Together these estimates prove the claim \eqref{eqn: sm diff}.

Without loss of generality we may assume that $\BB_2$ is centered at the south pole. Let $x_0$ be the center of $\BB_1$; from \eqref{eqn: sm diff} it follows that $d(x_0, o)\leq C\e^{1/2}$ where $o$ is the north pole and $d(\cdot , \cdot)$ denotes the distance on the sphere.  Let  $w_1$ be the truncated linear function corresponding to $u_1$ provided by Proposition~\ref{l:approx} (so $w$ is defined in polar coordinates by $w(r,\theta) = \beta r v(\theta)$) and let $\tilde{w}_1$ be its rotation so that $\{ \tilde{w}_1>0\}\cap \pa B_1$ is the hemisphere centered at the north pole. Since $|w_1| \leq C\| u_1\|_{W^{1,2}(B_1)}$, we see that 
\begin{align*}
\int_{B_1\setminus B\rho} | \tilde{w}_1 - u_1|^2 &   \leq C \e \| u\|_{W^{1,2}(B_1)}^ 2 + \int_{B_1\setminus B\rho}| \tilde{w}_1 - w_1|^2 \\
&\leq   C \e \| u\|_{W^{1,2}(B_1)}^ 2 + C \| u\|_{W^{1,2}(B_1)}^ 2 d(x_0, o)^2 \leq C \| u\|_{W^{1,2}(B_1)}^ 2\e .
\end{align*}
So, choosing $\tilde{w}_1$ and the truncated linear function $w_2$ corresponding to $u_2$ provided by Proposition~\ref{l:approx} completes the proof of the theorem.
\end{proof}

\subsection{Uniqueness of blow-ups when the ACF formula decays fast}\label{ssec: blowups}

Here we sketch the proof of Corollary \ref{c:uniqueblowup}. The argument gives a precise modulus for the convergence of the blow-ups as well.

\begin{proof}[Proof of Corollary \ref{c:uniqueblowup}.]
	We omit the subscripts $1, 2$ in $u_1, u_2$ below to simplify notation; any statement involving $u$ or $\b$ is valid for both $u = u_1$ and $u = u_2$, up to replacing $(\nu \cdot x)^+$ by $(\nu \cdot x)^-$ if $i = 2$.
	
	For $k=0,1, 2,\dots,$ consider the rescalings  $u^k(x) = \frac{u(2^{-k} x)}{2^{-k}}$. Recall that  $\w(r) = J(r) - J(0^+)$. We will show, by induction on $k$, that there are numbers $\b^k>0$ and unit vectors $\nu_k\in S^{n-1}$ such that
	\begin{equation}\label{e:inductiveh} 
		\int_{B_1 \sm B} \left|u^k - \b^k (\nu_k \cdot x)^+\right|^2 \leq C_1\, \w(2^{-k})
	\end{equation}
	and
	\begin{equation}\label{e:inductiveh2}
		 \left|\b^k - \b^{k-1}\right|^2 + \left|\nu_k - \nu_{k-1}\right|^2 \leq C_2\, \w(2^{-k-1}),
	\end{equation}
	where the constants $C_1$ and $C_2$ depend on $n$ and $\|u_i\|_{W^{1, 2}(B_1)}$ but are independent of $k$. For each $k$, the vector $\nu_k$ for $i=1$ will equal to $-\nu_k$ for $i=2$. For $k = 0$ this follows from a single application of Theorem \ref{t:stability}.
	
	Assume \eqref{e:inductiveh} and \eqref{e:inductiveh2} hold for all $j \leq k - 1$. Then we may apply Theorem \ref{t:stability} (with $\rho = 0$) to $u^k$, as $J[u_1^k, u_2^k]( r) = J[u_1,u_2]( 2^{-k}r) < 1 + \e$ by assumption:
	\[
	 \int_{B_1} |u^k - \b^k (\nu_k \cdot x)^+|^2 \leq C \w(2^{-k}) \sum_{i = 1}^2 \|u^k_i\|^2_{W^{1, 2}(B_1)}.
	\]
	The main point now is to estimate $\|u^k\|_{W^{1, 2}(B_1)}$ in a manner independent of $k$, and then show that $\b^k$ is close to $\b^{k-1}$ (and similarly for $\nu_k$).

	We apply the Caccioppoli inequality to $u^k$ to give
	\[
		\|u^k\|_{W^{1, 2}(B_1)} \leq C \|u^k\|_{L^2(B_2)},
	\]
	recalling they are subharmonic. From \eqref{e:inductiveh} with $k-1$, we have that
	\begin{align*}
		\|u^k\|_{L^2(B_2)} &= 2^{n/2+1} \|u^{k-1}\|_{L^2(B_1)} \\
		&\leq 2^{n/2+1} \|u^{k-1} - \b^{k-1} (\nu_{k-1} \cdot x)^+ \|_{L^2(B_1)} + C \b_i^{k-1}\\
		&\leq C\left[\sqrt{C_1 \w(2^{-k-1})} + \b^{k-1}\right]\\
		&\leq C\left[\sqrt{C_1 \e} + \b^{k-1}\right].
	\end{align*}
	Now applying \eqref{e:inductiveh2} for every $j \leq k-1$, we have
	\begin{align*}
		\b^{k-1} &\leq \b^0 + \sum_{j = 1}^{k - 1}\big|\b^{j} - \b^{j-1}\big|\\
		&\leq C\|u\|_{W^{1, 2}(B_1)} + \sqrt{C_2} \sum_{j = 1}^{k - 1} \sqrt{\w(2^{-j-1})}\leq C\Big(1 + \int_0^{2^{-k-1}} \frac{\sqrt{\w(r)}}{r}\Big) \leq C.
	\end{align*}
	Choosing $C_1$ large enough and then $\e$ small in terms of $C_1$ gives
	\[
		\|u^k\|^2_{W^{1, 2}(B_1)} \leq C[C_1 \e + 1] \leq \frac{C_1}{2}.
	\]
	This implies \eqref{e:inductiveh} for $k$.
	
	To control $\b^k$, we change variables:
	\[
		 \int_{B_{1/2}} \left|u^{k-1} - \b^k (\nu_k \cdot x)^+\right|^2 = 2^{- n-2}  \int_{B_1} \left|u^k - \b^k (\nu_k \cdot x)^+\right|^2 \leq C \w(2^{-k}).
	\]
	Then
	\[
		 \int_{B_{1/2}} \left|\b^k (\nu_k \cdot x)^+ - \b^{k-1} (\nu_{k-1} \cdot x)^+\right|^2 \leq C \w(2^{-k-1})
	\]
	by combining with \eqref{e:inductiveh} with $k-1$. Direct evaluation leads to
	\[
	|\b^k - \b^{k-1}|^2 + |\nu_k - \nu_{k-1}|^2 \leq C_2 \w(2^{-k-1}),
	\]
	and so \eqref{e:inductiveh2} holds for $k$.

	We are now in a position to conclude. From \eqref{e:inductiveh2} the numbers $\b^k$ and vectors $\nu^k$ have
	\[
		\sum_{k=0}^\infty |\b^k - \b^{k-1}| + |\nu_k - \nu_{k-1}| \leq C (\e + \int_0^1 \frac{\sqrt{\w(r)}}{r}dr) < \infty,
	\]
	so the sequences converge to some $\b_i, v$. Set
	\[
		\tilde{\w}(r) = \w(r) + \int_0^{2 r} \frac{\sqrt{\w(r)}}{r}dr,
	\]
	which a nondecreasing function tending to $0$ at $0$. Then from \eqref{e:inductiveh},
	\[
	 \int_{B_1} |u^k - \b (v \cdot x)^+|^2 \leq C \tilde{\w}(2^{-k}),
	\]
	and so changing variables,
	\[
		2^{(n + 2)k} \int_{B_{2^{-k}}} |u - \b (\nu \cdot x)^+|^2 \leq C \tilde{\w}(2^{-k}).
	\]
	This may be rewritten to give
	\[
	\frac{1}{r^{n+2}} \int_{B_{r}} |u - \b (v \cdot x)^+|^2 \leq C\tilde{\w}(2 r) \rightarrow 0,	
	\]
	concluding the proof.
\end{proof}

\section{Quantitative stability for the Faber-Krahn inequality}\label{sec: sperner}
 We now turn toward the proof of the quantitative Faber-Krahn inequality, Theorem~\ref{thm: quantitative sperner}. As we discussed in the introduction, the proof is essentially identical in the cases of the round sphere, Euclidean space, and hyperbolic space. To keep the notation from becoming unreasonably heavy, we will therefore present the proof in case of the round sphere, and include remarks throughout the proof guiding the reader through any modifications needed to generalize the proof to Euclidean or hyperbolic spaces. We give a few remarks about the notation in this section. Since the ambient Euclidean space does not play as prominent a role in this section, we will let $n$ denote the dimension of the $n$-dimensional sphere and thus avoid using $n-1$ throughout the section. Furthermore, in Section \ref{sec: 2d qacf} it was necessary to distinguish between a solid Euclidean ball and a spherical cap. In this section that is not necessary, so we will simply use $B$ to denote the spherical cap. This notation is also convenient for this section because - as already stated - the proofs for Euclidean space and hyperbolic space are very similar. Consequently, it is natural to use notation that easily adapts to the other two settings: for the proof in Euclidean space $B$ would simply be a solid Euclidean ball.  
 
All three simply connected space forms share the property that balls centered at {\it any } point minimize the first eigenvalue. It it clear, then, that any  stability statement must involve the distance to a nearby ball to $\Omega$. In \eqref{eqn: stability statement}, we chose to ``mod out'' by this symmetry by taking the infimum over all $x$ in the space of the distance to balls $B(x)$ centered at $x$. It turns out to be more convenient to prove the slightly stronger statement given  below as Theorem~\ref{thm: quantitative sperner restated}, in which we get rid of the translational invariance by choosing a particular ``set center'' of a given set $\Omega$, defined in the following way.

Consider the standard embedding of the round sphere $S^{n}(1) \subset \R^{n+1}$ and define the  {\it set center} $x_\Omega$  of $\Omega$ by setting
\[
x_\Omega = \frac{y_\Omega}{|y_\Omega|}
\]
whenever $0\neq y_\Omega = \fint_{\Omega} y \,d\vol_g(y) \in \R^{n+1}$. It is worth noting that there are other possible ways that one may define a notion of ``barycenter'' of a set  $\Omega \subset S^{n}$. Another possible notion is given in \cite{BDF17}. \\
 
Throughout this section, given a set $\Omega \subset S^{n}$, we let $u_\Omega$ denote its first eigenfunction extended by $0$ to be defined on all of $S^{n}$ and normalized so that $\| u_\Omega\|_{L^2(\Omega)}=1$.

\begin{theorem}\label{thm: quantitative sperner restated}
Fix $n\geq 1$ and $v \in (0, |S^{n}|)$. There exists a constant $c=c(n,v)$ such that the following holds. Let  $\Omega\subset S^{n}$ with $|\Omega|=v$ be a set for which $x_\Omega\in S^{n}$ is  defined.  Letting $\B$ denote the spherical cap centered at $x_\Om$ with $|\B|=v$, we have 
\begin{equation}\label{eqn: stability statement sec 3}
	\lambda_1(\Omega) - \lambda_1(\B) \geq c \left( |\Omega \Delta\B|^2 + \int_{S^{n}} |u_\Omega - u_\B|^2 \right) \,.
\end{equation}
\end{theorem}

\smallskip

\begin{remark}[Sets with  no barycenter]{\rm
	The statement of Theorem~\ref{thm: quantitative sperner restated} may look restrictive compared to Theorem~\ref{thm: quantitative sperner}. However, if $\Omega$ is a set for  which the barycenter is not  defined, then  necessarily $\lambda_1(\Omega) - \lambda_1(\B ) \geq c$, and  so \eqref{eqn: stability statement} holds  trivially in this  case. In particular, Theorem~\ref{thm: quantitative sperner restated} implies Theorem~\ref{thm: quantitative FK general} in the case of the round sphere.
	}
\end{remark}


\begin{remark}[Set centers on Euclidean space and hyperbolic space]
\rm{	On Euclidean space, an open set $\Omega \subset \R^n$ has a uniquely defined barycenter $x_\Omega = \fint_\Om x \,dx$. For an open set $\Omega \subset H^n$ in  hyperbolic space, the set center $x_\Omega = \text{argmin}_{x_0 \in H^n} \int_{\Omega} d(x, x_0)^2 \, d\vol(x) $ is well defined. The analogues of Theorem~\ref{thm: quantitative sperner restated} hold in each of these two spaces with $x_\Omega$ defined in this way; the proof will be identical up to the obvious modifications. Once more, the analogues of Theorem~\ref{thm: quantitative sperner restated} for Euclidean and hyperbolic space imply Theorem~\ref{thm: quantitative FK general} on these space forms.
	}
\end{remark}

We will work in spherical coordinates $(\theta, \phi)  \in S^{n-1} \times (0, \pi)$ on $S^{n}$, in which the round metric takes the form $g_{S^n} = d\phi^2 + \sin^{2}(\phi)d\theta^2$, where $d\theta^2$ is used to denote the round metric on $S^{n-1}$. A point $(\theta, \phi) \in S^{n}$ has geodesic distance $\phi$  to the north pole $o \in S^{n}$.  The Laplacian of a function $f$ in spherical coordinates $(\theta,\phi)$ is given by  
\begin{equation}\label{eqn: laplace in spherical}
 \Delta_{S^{n}} f(\theta,\phi) = (\sin \phi)^{1-n} \frac{\partial }{\partial \phi} \left((\sin \phi)^{n-1} \frac{\partial f}{\partial \phi}\right) + (\sin \phi)^{-2} \Delta_{S^{n-1}} f. 
\end{equation}

\begin{remark}[Coordinates on hyperbolic space and Euclidean space]{\rm 
On hyperbolic space, we similarly use polar coordinates $(\theta ,\phi)$, so that the metric is expressed in these coordinates by $g_{hyp} =d\phi^2 + \sinh^{2}(\phi) d\theta^2$, where once more we let $d\theta^2$ denote the round metric on $S^{n-1}$. A point $(\theta ,\phi ) \in H^n$ has geodesic distance $\phi$ from the distinguished point from which the polar coordinates are centered. The Laplacian of a function $f: H^n\to \R$ in these coordinates is given by 
\[
\Delta_{H^n} f(\theta, \phi) = (\sinh \phi)^{1-n} \frac{\partial }{\partial \phi} \left((\sinh \phi)^{n-1} \frac{\partial f}{\partial \phi}\right) + (\sinh \phi)^{-2} \Delta_{S^{n-1}} f.
\]
 On Euclidean space we use standard polar coordinates $(\theta ,\phi)$, where $\phi$ is the distance to the origin and the Laplacian of a function $f: \R^n \to \R$ is given by .
 \[
\Delta_{\R^n} f(\theta, \phi) =  \phi^{1-n} \frac{\partial }{\partial \phi} \left(\phi^{n-1} \frac{\partial f}{\partial \phi}\right) +  \phi^{-2} \Delta_{S^{n-1}} f.
\] 
 }\end{remark}
 
 Throughout this section, we will use the shorthand $\B_\rho = \B(o,\rho)$ to denote the spherical cap of radius $\rho$ centered at the north pole.
We will frequently make the identification of the boundary of a geodesic ball, $\pa \B(o,\rho)$, with the sphere $S^{n-1}$, noting that the induced metric on $\pa \B(o,\rho)$ is $\sin^2(\rho) d\theta^2$, where $d\theta^2$ is the round metric on $S^{n-1}$.

We proceed in two steps: in Section~\ref{sec: NSS}, we establish Theorem~\ref{thm: quantitative sperner restated} for sets that are small perturbations of a spherical cap, and in Section~\ref{sec: SP}, we call upon results from our companion paper \cite{AKN1} to conclude Theorem~\ref{thm: quantitative sperner restated} in the general case. 

\medskip

\subsection{Stability for nearly spherical sets}\label{sec: NSS}
The main goal of this section is to establish Theorem~\ref{prop: quantitative for nearly spherical sets} below, which is  Theorem~\ref{thm: quantitative sperner restated} restricted to a  class of sets in $S^{n}$ that are small $C^{2,\alpha}$ perturbations of spherical caps.

The class of {\it nearly spherical sets} given in the following definition will be our main objects of interest in this section.
\begin{definition} Given $\rho \in (0,\pi)$ and $\xi\in C^{2,\alpha}(\pa \B_\rho)$ with $\| \xi\|_{C^{2,\alpha}(\pa \B_\rho)} \leq \e,$ we define the following:
\begin{enumerate}
\item 	A set $\Omega\subset S^{n}$ is  parametrized by $\xi$ over $\B_\rho$ if  in spherical coordinates $\pa \Om$ takes the form
\begin{align*}	\label{eq: def: NSS}
	\pa \Omega  &= \{ (\theta, (1 + \xi(\theta))\rho) : (\theta,\rho) \in \pa \B_\rho\}.
\end{align*}

 \item A set $\Omega \subset S^{n}$ is a nearly spherical set if $\Omega$ is parametrized by $\xi$ over $\B_\rho$ such that  $|\Omega| = |\B_\rho|$ and the barycenter $x_\Om$ of $\Om$ is the north pole $o$.
\end{enumerate}
\end{definition}
Note that the barycenter of a set parametrized over $\B_\rho$  is always defined provided that $\e$ is taken to be sufficiently small depending on $n$ and $\rho$. 
The following theorem establishes quantitative stability  for Sperner's inequality among nearly spherical sets.
\begin{theorem}[Quantitative stability for nearly spherical sets]\label{prop: quantitative for nearly spherical sets}
	Fix $n\geq 2$ and $\g\in (0,\pi/2)$. There exist postitive constants $c$ and $\e$ depending only on $n$ and $\g$ such that the following holds. Let $|\rho -\pi/2|<\g$ and let $\Omega$ be a nearly spherical set parametrized by $\xi$ over $\B_\rho$ with $\|\xi\|_{C^{2,\alpha}(\pa \B_\rho)} \leq \e$. Then
	\begin{equation}\label{eqn: stability statement sec 3 repeat}
	\lambda_1(\Omega) - \lambda_1(\B_\rho) \geq c \left( |\Omega\Delta \B_\rho|^2 + \int_{S^{n}} |u_\Om - u_{\B_{\rho}}|^2 \right)\,.
\end{equation}
\end{theorem}

The basic idea of the proof of Theorem~\ref{prop: quantitative for nearly spherical sets} is the following. In Section~\ref{ssec: 2nd var}, we show that the deficit $\lambda_1(\Om) - \lambda_1(\B_\rho)$ of a nearly spherical set is equivalent to the second variation of $\lambda_1(\B_\rho)$, up to an error that can be made arbitrarily small by taking $\|\xi\|_{C^{2,\alpha}(\pa \B_\rho)}$ to be small; see  Theorem~\ref{lem: second var and deficit}. Section~\ref{ssec: gap} contains  the core of the proof of Theorem~\ref{prop: quantitative for nearly spherical sets}: in Theorem~\ref{t:gap} we establish a gap in the spectrum for the second variation operator.  As we discuss in the introduction, this spectral analysis is carried out through an implicit argument based on the maximum principle. Theorem~\ref{t:gap} is then applied in Theorem~\ref{th: deficit controls H1/2} to show that the deficit $\lambda_1(\Om) - \lambda_1(\B_\rho)$ of a nearly spherical set controls the $H^{1/2}$ norm of the parametrizing function $\xi$.  Finally, in Proposition~\ref{prop: control efunc} below of Section~\ref{ssec: h12 and distance}, we show that the right-hand side of \eqref{eqn: stability statement} is controlled linearly by $ \| \xi\|_{H^{1/2}(\pa B)}^2$.

Recall that the $H^{1/2}$ seminorm of a function $\xi:\pa \B_\rho \to \R$ 
is defined by 
$
\| \xi \|_{\dot{H}^{1/2}(\pa \B_\rho)}^2 = \int_{\B_\rho} |\na h_\xi|^2,
$
where $h_\xi$ is the harmonic extension of $\xi$, i.e. the unique solution to
\begin{equation}\label{eqn: harmonic extension}
\begin{cases}
	\Delta h_\xi = 0 & \text{ in }\B_\rho\\
	h_\xi = \xi & \text{ on } \pa \B_\rho.
\end{cases}
\end{equation}
The $H^{1/2}$ norm of a function $\xi:\pa \B_\rho \to \R$ is defined by $\| \xi\|_{H^{1/2}(\pa \B_\rho)}^2 = \| \xi\|_{L^2(\pa \B_\rho)}^2 + \| \xi \|_{\dot H^{1/2}(\pa \B_\rho)}^2$. \\

\medskip


\subsubsection{The deficit and the second variation}\label{ssec: 2nd var}
The first eigenfunction $u_{\B_\rho}$ of a spherical cap $\B_\rho$  is radially symmetric and decreasing. We let $u_\rho(\phi) = u_{\B_\rho}(\theta,\phi)$ and note that $|\na u_{\B_\rho}(\theta,\phi)|= -u_\rho'(\phi)= |u_\rho'(\phi)|.$ In the sequel, we will use the shorthand $|u_\rho'|$ to refer to number $|u_\rho'(\rho)|>0$ and will let $\lambda_\rho = \lambda_1(\B_\rho)$.

 Given a function $\xi \in H^{1/2}(\pa \B_{\rho})$ with $\int_{\pa \B_{\rho}} \xi =0$,  let $w_\xi:\B_\rho \to \R$ be the solution to 
 \begin{equation}\label{eqn: eigenvalue pb extension}
	\begin{cases}
		-\Delta w_\xi = \lambda_{\rho} w_\xi & \text{ in } \B_{\rho}\\
	\ \ \ \	\ w_\xi = |u_\rho'| \xi & \text{ on } \pa \B_{\rho}\\
		\int_{\B_\rho} w_\xi u_{\B_\rho}= 0.
	\end{cases}
\end{equation}
Such a solution exists and is unique by the Fredholm Alternative and is continuous by \cite{GT}. For such a function $\xi$, the second variation $\delta^2 \lambda_1(\B_\rho)[\xi, \xi ]$ of $\lambda_1$ at $\B_\rho$ in the direction $[\xi, \xi]$ is given by
\begin{equation}\label{eqn: second variation}
		\delta^2 \lambda_1(\B_\rho)[\xi, \xi ] = 2 \int_{\B_\rho} \left(|\na w_\xi|^2 - \lambda_{\rho} w_\xi^2\right) \, dx +  \mathcal{H}_{\B_\rho} |u_\rho'|^2\int_{\pa \B_\rho}  \xi^2 \, dx;
\end{equation} 
see \cite[Lemma 2.8]{dl19} or \cite[Appendix A]{BDV15} for similar computations; note that on $\R^n$ this expression matches  \cite[Lemma 2.8]{dl19} since we take  $\xi$ with mean zero and since the ball has constant mean curvature. Here $\mathcal{H}_{\B_\rho} $ is the scalar mean curvature of $\pa \B_{\rho}$ with the sign convention that $\mathcal{H}_{\B_\rho} >0$ for $\rho < \pi/2$ and $\mathcal{H}_{\B_\rho} <0 $ for $\rho \in (\pi/2, \pi).$ On Euclidean space and hyperbolic space, the second variation takes the analogous form.

Let $L_\rho:H^{1/2}(\pa \B_\rho)\cap \{ \int_{\pa \B_\rho} \xi = 0\} \to L^2(\pa \B_\rho)$ denote the Dirichlet-to-Neumann map corresponding to the problem \eqref{eqn: eigenvalue pb extension} translated by the zero order term $\frac{1}{2}\mathcal{H}_{\B_\rho} |u_\rho'|^2$. More specifically, given $\xi \in H^{1/2}(\pa \B_\rho)$ with $\int_{\pa \B_\rho} \xi =0$, define
\begin{equation}\label{eqn: def DtoN}
L_{\rho}\xi = \pa_\nu w_\xi + \frac{1}{2}\mathcal{H}_{\B_\rho} |u_\rho'|^2 \xi,
\end{equation}
where $w_\xi$ is the unique solution to \eqref{eqn: eigenvalue pb extension} and $\nu$ is the outer unit normal to $\B_\rho$. In this way, we can express the second variation on the ball as 
\[
\delta^2 \lambda_1(\B_\rho)[\xi, \xi ] = 2 \int_{\pa \B_\rho} \xi L_{\rho}\xi .
\]
The following theorem shows that the deficit in Sperner's inequality for a nearly spherical set is equal to the second variation, up to an error depending on $\|\xi\|_{C^{2,\alpha}(\pa \B_\rho)}$.
\begin{theorem}\label{lem: second var and deficit} Fix $\g\in(0,\pi/2)$. 
 There is a modulus of continuity $\om$ depending only on $n$ and $\g$ such that the following holds. Fix $\rho$ with $|\pi/2-\rho|<\g$ and let $\Omega \subset S^{n}$ be a set parametrized by $\xi$ over $\B_\rho$ with $|\Om| =|\B_\rho|.$ Setting $\hat{\xi} = \xi - \int_{\pa \B_\rho} \xi,$ we have
	\begin{equation}
		\lambda_1(\Omega) - \lambda_1(\B_\rho) = \frac{1}{2}\delta^2\lambda_1(\B_\rho)\big[\hat{\xi},\hat{\xi}\big] +\om\left(\|\hat{\xi}\|_{C^{2,\alpha}(\pa \B_\rho)}\right) \| \hat{\xi} \|_{H^{1/2}(\pa \B_\rho)}^2.
	\end{equation}
\end{theorem}
Theorem~\ref{lem: second var and deficit} was established in the Euclidean setting in \cite[Theorem 1]{d02} (see also \cite[Theorem 1.4]{dl19} for the $C^{2,\alpha}$ norm replaced by the $W^{2,p}$ norm for $p>n$). Adapting the proof to the spherical and hyperbolic settings requires only technical modifications, which we omit.

\begin{remark}\label{rmk: volume constraint}
	{\rm 
	For a function $\xi \in H^{1/2}(\pa \B_\rho)$ with nonzero mean, the Fredholm Alternative implies that no solution exists to the problem \eqref{eqn: eigenvalue pb extension} (even without the orthogonality constraint $\int_{\B_\rho} w_\xi u_{\B_\rho} =0$), and for this reason the second variation is only defined for mean zero functions.   However, for a nearly spherical set  $\Omega$, a Taylor expansion of the volume constraint  shows that 
\[
0 = |\Omega| - |\B_\rho| = \int_{\pa \B_{\rho}} \xi + o(\|\xi\|_{L^2(\pa \B_{\rho})});
\]
see, for instance, \cite{BDF17}. This implies that $\| \xi -\hat{\xi}\|_{L^2(\pa \B_\rho)} = o(\|\xi\|_{L^2(\pa \B_\rho)})$. In particular, $\| \xi\|_{H^{1/2}(\pa \B_\rho)} = \| \hat{\xi}\|_{H^{1/2}(\pa \B_\rho)} + o(\|\xi\|_{L^2(\pa \B_\rho)})$ and 
\[
\frac{1}{2}\|\xi\|_{C^{2,\alpha}(\pa \B_\rho)} \leq \|\hat{\xi}\|_{C^{2,\alpha}(\pa \B_\rho)} \leq 2\|\xi\|_{C^{2,\alpha}(\pa \B_\rho)}.
\] Therefore, the conclusion of Theorem~\ref{lem: second var and deficit} above may be equivalently written as 
	\begin{equation}
		\lambda_1(\Omega) - \lambda_1(\B_\rho) = \frac{1}{2}\delta^2\lambda_1(\B_\rho)\big[\hat{\xi},\hat{\xi}\big] +\om\left(\|\xi\|_{C^{2,\alpha}(\pa \B_\rho)}\right) \| \xi \|_{H^{1/2}(\pa \B_\rho)}^2.
	\end{equation}
	}
\end{remark}
 
\medskip

\subsubsection{The spectral gap}\label{ssec: gap}
We study the spectrum of the shifted Dirichlet-to-Neumann operator $L_{\rho}$ introduced in \eqref{eqn: def DtoN} and establish a spectral gap in Theorem~\ref{t:gap} below. In contrast to the Euclidean case, the spectrum cannot be computed explicitly and therefore the spectral analysis relies on implicit methods based on the maximum principle.

We let $\{Y_i\}_{i=0}^\infty$ be the orthonormal basis for $L^2(\pa \B_\rho) $ of spherical harmonics, i.e. solutions to the eigenvalue equation  $\Delta_{S^{n-1}} Y_i= -\mu_i Y_i$, or equivalently $  \Delta_{\pa \B_\rho} Y_i = -\sin^{-2}(\rho)\mu_i Y_i$, on $\pa \B_\rho$ with each $Y_i$ normalized so that $\int_{\pa \B_\rho} Y_i^2 =1.$ The eigenfunctions are ordered such that $\mu_i \leq \mu_{i+1}$ for all $i$.  Recall that
\begin{equation}\label{eqn: laplace eigenvalues}
\begin{split}
 \mu_0 &= 0\\
 \mu_1 & =\dots = \mu_{n} = (n-1)\\
 \mu_{n+1} &= 2n.	
\end{split}
 \end{equation}
 The unique spherical harmonic corresponding to $\mu_0=0$ is constant, and thus for each $i\geq 1$ we have
     \begin{equation}  \label{e:fredholm}
   \int_{\pa \B_\rho} Y_i =0.
  \end{equation}

We let $w_i$ denote the solution of \eqref{eqn: eigenvalue pb extension} with $\xi = Y_i$. That is, $w_i\in H^1(\B_\rho)\cap C(\B_{\rho})$  is the unique solution to
  \begin{equation}\label{eqn: eigenvalue pb extension Yi}
	\begin{cases}
		-\Delta w_i = \lambda_{\rho} w_i & \text{ in } \B_{\rho}\\
		\ \ \ \ w_i = |u_\rho'|Y_i & \text{ on } \pa \B_{\rho}\\
		\int_{\B_\rho} w_i u_{\B_\rho}  = 0.
	\end{cases}
\end{equation}
 The following lemma shows that the solutions $w_i$ take a particular form that will be crucial to our spectral analysis.
\begin{lemma}\label{lem: efns}
	For each $i\geq 1$, the unique solution $w_i$ of \eqref{eqn: eigenvalue pb extension Yi} takes the form
	\begin{equation}
		\label{eqn: wi form}
		w_i(\phi, \theta)=Y_i(\theta)g_i(\phi)
	\end{equation}
	in spherical coordinates, and $g_i$ satisfies the following three properties:
 \[
  \begin{aligned} 
    &(1) \quad g_i(0)=0 \\
    &(2) \quad g_i(\rho) =|u_\rho'|\\
    &(3) \quad 
    g_i(\phi)>0 \text{ for } 0<\phi <\rho.
   \end{aligned}
 \]
 \end{lemma}
\begin{proof}
Fix $i \geq 1$ and the corresponding spherical harmonic $Y_i$ and solution $w_i$ of \eqref{eqn: eigenvalue pb extension Yi}. We omit the subscript $i$ in the remainder of this proof for notational simplicity.
 Let us make the ansatz that $w$ takes the form \eqref{eqn: wi form} in spherical coordinates. Then $w(\theta, \rho) =Y(\theta) g(\rho)$ on $\pa \B_\rho$, so $w$ satisfies the boundary condition in \eqref{eqn: eigenvalue pb extension Yi} provided $g(\rho)=|u_\rho'|.$ Furthermore, thanks to the orthogonality \eqref{e:fredholm}, we have
\begin{align*}
\int_{\B_\rho } w\, u_{\B_{\rho}} & = \int_0^{\rho} \left( \int_{S^{n-1}} Y(\theta)\,d\mathcal{H}^{n-1}(\theta)\right) u_\rho(\phi) g(\phi)\, \sin(\phi)^{n-1} \, d\phi =0,
\end{align*}
so $w$ satisfies the orthogonality condition in \eqref{eqn: eigenvalue pb extension Yi}.  Finally, from the expression \eqref{eqn: laplace in spherical}
for the Laplacian of a function $f$ in spherical coordinates $(\theta,\phi)$, we see that  $g(\phi)$ satisfies the ordinary differential equation
\begin{equation}  \label{e:legendre}
 \frac{d^2 g}{d \phi^2} + (n-1) \frac{\cos \phi}{\sin \phi} \frac{d g}{d \phi} + \left(\lambda_{\rho} - \frac{\mu}{\sin^2 \phi} \right)g =0.
\end{equation}
Under the change of variable $t=\cos \phi$, the solution $g$ is a solution to the associated Legendre equation, which are well-known special functions.
Thus, up to choosing a constant multiple of $g$ so that the boundary condition is satisfied, this function is the unique $H^1(\B_\rho)$ solution of \eqref{eqn: eigenvalue pb extension}.

Next we verify properties (1)--(3). Property (2) holds by our choice of constant multiple of $g$. As a solution of \eqref{eqn: eigenvalue pb extension Yi}, $w$ is nonconstant and continuous, so it follows that $g(0)=0$ and thus $(1)$ holds. Finally, to show (3), we suppose by way of contradiction that $g(\rho_1)=0$ for some $0<\rho_1<\rho$. Then we see that $w\in H^1_0(\B_{\rho_1})$ solves 
  \[
    \frac{\int_{\B_{\rho_1}} |\nabla w|^2}{ \int_{\B_{\rho_1}} w^2} = \lambda_{\rho}.
  \] 
  Thus, $\lambda_{\rho_1} \leq \lambda_{\rho}$; however, since $\B_{\rho_1} \subset \B_{\rho}$ we have that $\lambda_{\rho} < \lambda_{\rho_1}$. This gives a contradiction and we see that $g$ is nonvanishing in $(0,\rho)$. Together with (2) this proves (3).
  \end{proof}
Using Lemma~\ref{lem: efns}, we show in the following corollary that the spherical harmonics for $i\geq 1$ are eigenfunctions of the Dirichlet-to-Neumann operator \eqref{eqn: def DtoN}.
\begin{corollary}\label{cor: eigenfunctions}
	Each spherical harmonic $Y_i$ for $i\geq 1$ is an eigenfunction of the Dirichlet-to-Neumann operator $L_{\rho}$ defined in \eqref{eqn: def DtoN} with eigenvalue $\eta_i$ given by
	\begin{equation}
		\label{eqn: eta def}
		\eta_i = g_i'(\rho)+\frac{1}{2}\mathcal{H}_{\B_\rho} |u_\rho'|^2,
	\end{equation}
where $g_i$ is the function given in Lemma~\ref{lem: efns}. If $\mu_i = \mu_{i+1}$, then $\eta_i = \eta_{i+1}$. Furthermore, $\eta_i \geq0 $ for all $ i \in \mathbb{N}$ and $\eta_i =0$ for $i=1,\dots , n$. 
\end{corollary}
\begin{proof}
Making use of Lemma~\ref{lem: efns}, we compute that $L_{\rho}Y_i = g_i'(\rho) Y_i + \frac{1}{2}\mathcal{H}_{\B_\rho} |u_\rho'|^2Y_i $ where $g_i$ is the solution to \eqref{e:legendre} found in Lemma~\ref{lem: efns}. In particular,  $Y_i$ is an eigenfunction of $L_{\rho}$ with eigenvalue $\eta_i$ as in \eqref{eqn: eta def}. If $\mu_i = \mu_{i+1}$, then we see from \eqref{e:legendre} that $g_i = g_{i+1}$ and thus $\eta_i = \eta_{i+1}$ by \eqref{eqn: eta def}.

The next two claims follow from  geometric considerations. Suppose by way of contradiction that $\eta_i<0$ for some $i\geq 1$. For $t>0$ and small, let $\Om_t $ be the set parametrized by $\xi_t$ over $\B_\rho$ where $\xi_t  = tY_i + s(t) Y_0$ and $s(t)$ is the function defined such that $|\Om_t|= |\B_\rho|$.  By Sperner's inequality and Theorem~\ref{lem: second var and deficit},
\begin{align*}
	0 \leq \lambda_1(\Omega_t) - \lambda_1(\B_\rho) &= t^2 \eta_i + t^3\om\left (\| Y_i\|_{C^{2,\alpha}(\pa \B_\rho)}\right)\|Y_i\|_{H^{1/2}(\pa \B_\rho)}^2 \leq \frac{t^2}{2}\eta_i <0,
\end{align*}
where the penultimate inequality holds for $t$ sufficiently small depending on $Y_i$. We reach a contradiction and conclude that $\eta_i \geq 0$ for all $i\geq 1$.

We sketch the proof that $\eta_i = 0$ for $i=1,\dots, n$. Note that $\eta_1=\dots= \eta_{n-1}$ by \eqref{eqn: laplace eigenvalues} so it suffices to show that $\eta_1=0$. Let $e_1 \dots, e_{n+1} $ be an orthonormal basis for $\R^{n+1}$. Consider the standard embedding of $S^{n}$ in $\R^{n+1}$ such that the north pole is $e_{n+1}$. The spherical harmonic $Y= Y_1$ is  the restriction to $\pa B_\rho$ of the linear function $x \cdot e_1$ on $\R^{n+1}$. 

For $t \in (-\e ,\e)$, let $x_t = \exp(te_1)$ and let $\Omega_t = B(x_t, \rho)$.  For $t$ sufficiently small, there is a smooth and uniquely defined one-parameter family of functions $\xi_t$ such that 
\[
\pa \Om_t  = \{ (\theta, (1+ \xi_t(\theta))\rho ) \ : \ (\theta , \rho) \in \pa B_\rho\}
\] 
in spherical coordinates. Moreover, for $t$ sufficiently small, $\xi_t  = \sum_{i=0}^\infty a_i(t) Y_i $
with 
\[
a_1(t) =t/\rho, \qquad \text{ and } \qquad \frac{d}{dt}\big|_{t=0} a_i(t) = \lim_{t\to 0} a_i(t) /t=  0 \text{ for } i \neq 1.
\]

Because $\lambda_1(\Omega_t) =\lambda_1(B_\rho)$ for all $t$, we have, in particular, 
\begin{equation}\label{eqn: 2nd der} \frac{d^2}{dt^2} \lambda_1 (\Omega_t) =0	
\end{equation}
 for all $t$ sufficiently small.
The expression for second variation at $\Om_t$ changes smoothly with respect to $t$. So, dividing \eqref{eqn: 2nd der} by $t^2$ and letting $t \to 0$, we use the expression  \eqref{eqn: second variation} to arrive at 
\[
0 = \lim_{t \to 0}\  \frac{1}{t^2} \left(\frac{d^2}{dt^2} \lambda_1(\Omega_t)\right) = \lim_{t \to 0} \ \sum_{i=1}^\infty  \eta_i \left(\frac{a_i(t)}{t}\right)^2 = \eta_1.
\]
This concludes the proof of the corollary.
%
%
%
\end{proof}

Corollary~\ref{cor: eigenfunctions} shows that $Y_1,\dots, Y_{n}\in \ker L_\rho.$
The following theorem is the key point in our spectral analysis: we establish a spectral gap to show that the kernel of $L_\rho$ consists {\it only} of the span of $\{Y_i\}_{i=1}^{n}$.

\begin{theorem}[Spectral gap]  \label{t:gap} Fix $n\geq 2$ and $\rho \in (0,\pi)$.
Let $\{\eta_i\}_{i=1}^\infty$ be the eigenvalues of $L_\rho$ given in \eqref{eqn: eta def}.  Then $\eta_{n+1}>\eta_{n}=0$ and so 
\[
\eta_i \geq \eta_{n+1} >0 \qquad \text{ for all }i \geq n+1.
\]
 Furthermore, given any $\g \in (0,\pi/2)$, there exists $\bar \eta$ depending only on $n$ and $\g$ such that $\eta_{n+1} \geq \bar \eta $ for $|\rho-\pi/2|\leq \g$.
 \end{theorem}

\begin{proof}[Proof of Theorem~\ref{t:gap}]
 As noted in Corollary~\ref{cor: eigenfunctions}, if $\mu_i=\mu_{i+1}$, then $\eta_i = \eta_{i+1}$.
 Let $w_i = Y_i(\theta) g_i(\phi)$ be the solution to \eqref{eqn: eigenvalue pb extension}. We consider $h=g_{n} - g_{n+1}$.   From properties (1) and (2) in Lemma~\ref{lem: efns},   we have that $h(\rho)=h(0)=0$. Furthermore, from \eqref{e:legendre} and property (3) of Lemma~\ref{lem: efns}, we see that $h$ is a strict supersolution to the equation that $g_{n}$ satisfies. That is,
  \begin{equation} \label{e:supersolt}
   \frac{d^2 h}{d \phi^2} + (n-1) \frac{\cos \phi}{\sin \phi} \frac{d h}{d \phi} + \left(\lambda_{\rho} - \frac{\mu_{n}}{\sin^2 \phi} \right)h = 
    \frac{\mu_{n} - \mu_{n+1}}{\sin^2 \phi} g_{n-1}(\phi) <  0.
 \end{equation}
 We will show that $h>0$ for $\phi \in (0,\rho)$ and $h'(\rho)>0$. Notice that in  the case of the hemisphere $\rho=\pi/2$, we have  $\lambda_{\rho}=n$
 and $\mu_{n+1}=2n$ (recall \eqref{eqn: laplace eigenvalues}), hence  the zero-order term 
 $(\lambda_{\pi/2} - \mu_{n+1} \sin^{-2}(\phi))\leq 0$ for all $\phi$. Moreover, there exists $\epsilon$ depending on dimension $n$, such that for  $\rho \in (\pi/2-\epsilon, \pi)$, the zero-order term 
 $(\lambda_{\rho} - \mu_{n+1} \sin^{-2}(\phi))\leq 0$ for all $\phi.$ 
  From $\eqref{e:supersolt}$ we have $h_{\rho}$ is a strict supersolution, and from the comparison principle we conclude that $h>0$, and so $g_{n}< g_{n+1}$ on $(0,\rho)$. Furthermore, from the Hopf principle we conclude that $h'(\rho)<0$. 
 
 We will now show the same two properties for any $\rho \in (0,\pi)$. If
  $h_{\rho}$ is the function associated with $\rho$, then define the set  
 \[
 A:= \{\rho \in (0,\pi) \mid h_{\rho}>0 \text{ in }  (0,\rho)\}.
 \]
 We will show $A$ is both relatively open and closed in $(0,\pi)$. We have shown already that $A \supset (\pi/2-\epsilon, \pi)$, and therefore $A$ is nonempty. Since $h_{\rho}$ is the difference of two 
 Legendre functions defined on the whole interval $(0,\pi)$, then $h_{\rho}$ varies continuously with $\rho$ in $C^2$ on $[0,\phi_0]$ for any $\phi_0 < \pi$.  
 Suppose  $\rho_k \in A$ and $\rho_k \to \rho \in (0,\pi)$. Then by uniform convergence $h_{\rho}\geq 0$. If $h_{\rho}(\phi)=0$, then $\phi$ is a local minimum, and so $h_{\rho}'(\phi)=0$ and $h_{\rho}''(\phi)\geq0$. But from \eqref{e:supersolt} we have $h_{\rho}''(\phi)<0$, which is a contradiction. Therefore, $h_{\rho}(\phi)>0$, and so $h_{\rho} \in A$. This proves that $A$ is closed.  
 
 To show $A$ is also open, we will use the fact (shown below) that if $h_{\rho} \geq 0$, then $h'(\rho)<0$. Suppose by way of contradiction that $\rho \in A$
 and there exists a sequence $\rho_k \to \rho$ and $\phi_k\in (0,\rho_k)$ such that $h_{\rho_k}(\phi_k)\leq 0$. Since $h_{\rho_k}$ converges uniformly to $h_{\rho}$, it follows that 
 for a subsequence, either $\phi_k \to \rho$ or $\phi_k \to 0$. If $\phi_k \to \rho$, then from the uniform convergence, there also exists a second sequence $\phi_{1,k} \to \rho$, such that 
 $h_{\rho_k}(\phi_{1,k})=0$. From the mean value theorem, there exists a third sequence $\phi_{2,k} \to \rho$, such that $h'_{\rho_k}(\phi_{2,k})=0$. From the $C^1$ convergence up to the boundary point $\rho$, we would then have that $h'_{\rho}(\rho)=0$ which is a contradiction. If instead $\phi_k \to 0$, then from the uniform convergence, there exists a second sequence of local minima $\phi_{1,k} \to 0$. But for $\phi$ close to zero, the zero-order term $(\lambda_{\rho_k} - \mu_{n+1}\sin^{-2}(\phi))\leq 0$, so from the comparison principle, we have a contradiction. We thus conclude that $A=(0,\pi)$.  
 
  We now show the claim that $h'(\rho)<0$. This claim is not only necessary for the proof above, but we will also use it to show the eigenvalue gap. We only need to consider the situation when $(\lambda_{\rho} - \mu_{n+1}\sin^{-2}(\rho))>0$ and necessarily $\rho \in (0,\pi/2 - \epsilon)$.  Under appropriate conditions one may still apply the Hopf principle; however, in our one-dimensional case, we give the following simpler argument. 
 
  Since $h\geq 0$ and the zero-order term  is positive in $(\rho - \delta, \rho)$ we have that 
  \begin{equation}\label{eqn: new eqn}
   h'' + (n-1) \cot \phi h' = \zeta(\phi)\leq 0 \quad \text{ for } \phi \in (\rho - \delta, \rho).  
  \end{equation}
Assume by way of contradiction that $h'(\rho)= 0$. Then applying the explicit representation formula for first order linear equations to \eqref{eqn: new eqn} together with the initial condition $h'(\rho)=0$, we have that 
  \[
   h'(\phi)= \frac{1}{p(\phi)} \int_{\rho}^\phi p(s) \zeta(s) \ ds \geq 0 \quad (\text{since } \zeta \leq 0 \text{ and } \phi < \rho),
  \]
   and where $p(s)>0$ is the integrating factor given by
  \[
   p(s) = \text{Exp}\left\{\int_{\rho}^{s} (n-1) \cot t  \ dt \right\}. 
  \]
  Then $h' \geq 0$ in $(\rho - \delta, \rho)$. Since $h\geq 0$, we must then have that $h'\equiv 0$, so that $h \equiv 0$ on $(\rho - \delta, \rho)$. But this is a contradiction since
  as already noted $h$ is a strict supersolution to the equation that $g_{n+1}$ satisfies. We must then conclude that $h'(\rho)<0$. 
  
   Consequently, we have 
 $g_{n+1}'(\rho)> g_{n}'(\rho)$ and thus $\eta_{n+1} >\eta_{n}$. In particular, $\eta_i \geq \eta_{n+1} >0$ for all $i\geq n+1$. 
Finally, from \eqref{e:legendre}, we see that $\eta_{n+1}$ depends continuously on $\rho$, and so $\eta_{n+1}=\eta_{n+1}(\rho) \geq \bar \eta$ for $\rho$ in any compact subset $ [\pi/2-\g, \pi/2+\g]$ of $(0,\pi)$.
\end{proof}

\smallskip
Theorem~\ref{t:gap} will allow us to easily show that the second variation controls the $L^2$ norm squared of any function $\xi \in H^{1/2}(\pa \B_\rho)$ that is orthogonal to $Y_0, \dots , Y_{n-1}$; see the proof of Corollary~\ref{cor: second var controls h1/2} below. In order to improve this $L^2$ control to $H^{1/2}$ control, we need the following lemma.
 
 \begin{lemma}\label{lem: l2 bound} Fix $n\geq 2$ and $\g\in (0,\pi/2)$. There exists a constant $C$ depending only on $n$ and $\g$ such that the following holds. Let $|\pi/2-\rho|\leq\g$ and let $\xi \in H^{1/2}(\pa \B_{\rho})$ with $\int_{\pa \B_\rho} \xi =0$. For the solution $w_\xi$  of \eqref{eqn: eigenvalue pb extension}, we have
	$\|w_\xi\|_{L^2(\B_{\rho})} \leq \| \xi\|_{L^2(\pa \B_\rho)}.$
\end{lemma}  
\begin{proof}
Let $w=w_\xi$ for notational simplicity. 
Let $\vphi$ be the unique solution to the following auxiliary problem:
\[
\begin{cases}
	(-\Delta - \lambda_\rho) \vphi  = w & \qquad \text{ on } \B_\rho\\
	\vphi  =0 & \qquad \text{ on } \pa {\B_\rho}\\
	\int_{\B_\rho} \vphi u_{\B_\rho} = 0.
\end{cases}
\]
Such a solution exists and is unique by the Fredholm alternative. We multiply the equation by $w$ and integrate by parts twice to find
\begin{align*}
	\int_{\B_\rho} w^2 &= \int_{\B_\rho} w (-\Delta - \lambda_\rho) \vphi  \\
	&=  - \int_{\pa \B_\rho} \pa_\nu \vphi w  +\int_{\pa \B_\rho} \pa_\nu w \vphi+ \int_{\B_\rho} \vphi  (-\Delta - \lambda_\rho)w =  -\int_{\pa \B_\rho} \pa_\nu \vphi w.
\end{align*}
Then by H\"{o}lder's inequality and the trace embedding (see \cite[Theorem 5.22]{Adams}) respectively, we see that 
\[
\begin{aligned}
\int_{\B_\rho} w^2 &\leq \| \pa_\nu\vphi\|_{L^2(\pa \B_\rho)}|u'_\rho|\| \xi\|_{L^2(\pa \B_\rho)} \\
&\leq C \| \na \vphi\|_{H^{1}(\B_\rho)}|u'_\rho|\| \xi\|_{L^2(\pa \B_\rho)} \leq C\| \vphi\|_{H^2(\B_\rho)}|u'_\rho|\| \xi\|_{L^2(\pa \B_\rho)},
\end{aligned}
\]
where the constant $C>0$ depends on $n$ and $\rho$, but is uniformly bounded above for $|\pi/2-\rho|\leq \g.$
By standard elliptic estimates (see \cite[Theorem 9.13]{GT}), $\| \vphi\|_{H^2(\B_\rho)} \leq C\|w\|_{L^2(\B_\rho)}$, where again $C>0$ is uniformly bounded above depending on $n$ and $\g$. Since additionally $|u'_\rho|$ is bounded by a constant $C=C(n,\gamma)$,  we find that 
\[
\int_{B\rho} w^2 \leq C\|w\|_{L^2(\B_\rho)} \| \xi\|_{L^2(\pa \B_\rho)}.
\]
Dividing through by $\|w\|_{L^2(\B_\rho)}$ establishes the claim.
\end{proof}

From Theorem~\ref{t:gap} and Lemma~\ref{lem: l2 bound}, we can conclude that the second variation controls $\|\xi\|_{H^{1/2}(\pa B)}^2$ for any function $\xi \in H^{1/2}(\pa \B_\rho)$ that is orthogonal to $Y_0, \dots , Y_{n-1}$:
\begin{corollary}\label{cor: second var controls h1/2} 
Fix $n\geq 2$ and $\g \in (0,\pi/2)$. There is a positive constant $\hat{\eta}$ depending only on $n$ and $\g$ such that the following holds. Let $\rho \in (0,\pi)$ satisfy $|\pi/2-\rho| \leq\g.$ For any $\xi \in H^{1/2}(\pa \B_\rho)$ with $\int_{\pa \B_\rho} \xi Y_i = 0$ for $i=0,\dots, n,$ we have 
 \begin{equation}\label{eqn: second var and h1/2}
 \delta^2\lambda_1(\B_\rho)[\xi,\xi]\geq \hat{\eta} \|\xi\|_{H^{1/2}(\pa \B_\rho)}^2.
 \end{equation}
\end{corollary}
\begin{proof}
We express $\xi = \sum_{i={n+1}}^\infty a_i Y_i$ in the basis of spherical harmonics.
 By the linearity of $L_\rho$, we see that
$
 L_{\rho} \xi  = \sum_{i={n+1}}^\infty \eta_i a_iY_i. 
$
Thus, Theorem~\ref{t:gap} shows that 
\begin{equation}\label{eqn: 2nd var in basis}
 \frac{1}{2}\delta^2 \lambda_1(\B_\rho)[\xi, \xi ] =  \sum_{i={n+1}}^\infty \eta_i  a_i^2 \geq \eta_n \sum_{i={n+1}}^\infty  a_i^2 = \bar{\eta} \|\xi\|_{L^2(\pa \B_{\rho})}^2,
 \end{equation}
where $\bar \eta$ depends only on $n$ and $\gamma$.
In order to improve this estimate to replace $\|\xi\|_{L^2(\pa \B_\rho)}$ with  $\| \xi\|_{H^{1/2}(\pa \B_\rho)}$, let $h_\xi$ denote the harmonic extension of $|u_\rho'| \xi$ as defined in \eqref{eqn: harmonic extension} 
 and note that 
\[
 |u_\rho'|^2 \| \xi \|_{\dot{H}^{1/2}(\B_\rho)}^2 = \int_{\B_\rho} | \na h_\xi|^2  \leq \int_{\B_\rho} |\na w_\xi|^2 .
\]
Therefore, by first using the expression  \eqref{eqn: second variation} for the second variation, and then applying Lemma~\ref{lem: l2 bound} to $w_\xi$, we have 
\begin{align*}
2|u_\rho'|^2 \| \xi \|_{\dot{H}^{1/2}(\B_\rho)}^2 &  \leq 2\int_{\B_\rho} |\na w_\xi|^2   \\
&\leq  \delta^2\lambda_1(\B_\rho)[\xi,\xi]   + |\mathcal{H}_{\B_\rho}||u_\rho'|^2 \int_{\pa \B_\rho} \xi^2+2\lambda_\rho\int_{\B_\rho} w_\xi^2\\
& \leq  \delta^2\lambda_1(\B_\rho)[\xi,\xi] + C\int_{\pa \B_\rho} \xi^2,
\end{align*}
where $C= C(n,\gamma)$. Noting that $|u_\rho'|>0$ depends continuously on $\rho$ and thus $| u_\rho'| \geq c({n,\gamma})>0$,  we apply the estimate \eqref{eqn: 2nd var in basis} to the second term on the right-hand side to conclude the proof. 
\end{proof}

Finally, we can combine Corollary~\ref{cor: second var controls h1/2} with Theorem~\ref{lem: second var and deficit} to obtain the following theorem.

\begin{theorem}\label{th: deficit controls H1/2}
	Fix $n\geq 2$ and $\g \in (0,\pi/2)$. There are positive constants $\e$ and $\hat{\eta}$ depending only on $n$ and $\g$ such that the following holds. Let $|\pi/2-\rho| \leq \g$ and let $\Omega$ be a nearly spherical set parametrized by $\xi$ over $\B_\rho$ with $\|\xi\|_{C^{2,\alpha}(\pa \B_\rho)}\leq \e.$ Then 
		\[
	\lambda_1(\Omega)-\lambda_1(\B_\rho) \geq \hat{\eta} \| \xi\|_{H^{1/2}(\pa \B_{\rho})}^2.
	\]
\end{theorem}
\begin{proof}
We express $\xi$ in the basis of spherical harmonics as $\xi = \sum_{i=0}^\infty a_i Y_i$ where $a_i = \int_{\pa \B_\rho} Y_i \xi.$   The volume constraint and the set center constraint respectively imply that  
	\begin{align}\label{e: constraints linearized}
		a_0^2 = o\left( \|\xi\|_{L^2(\pa \B_\rho)}^2\right), \qquad a_1^2, \dots a_{n}^2	= o\left( \|\xi\|_{L^2(\pa \B_\rho)}^2\right)
	\end{align} 
	with the $o(\cdot)$ term depending on $n$ and $\gamma$.
Indeed, as we noted in Remark~\ref{rmk: volume constraint}, the first fact is a standard computation that can be found in \cite[Theorem 3.1]{BDF17}. To see the second, consider the standard embedding of $S^{n} \subset \R^{n+1}$ with the north pole $o =e_{n+1}$, and fix a basis vector $e_i$ for $i=1,\dots, n$. In spherical coordinates we may write $e_i = (1,\bar \theta, \pi /2) $ for some $\bar \theta \in S^{n-1}$. So, from the barycenter constraint $x_\Omega = e_{n+1}$, we have  	
\begin{align*}
 0=e_i  \cdot \int_\Omega x \, d\vol& = \int_{S^{n-1}} \int_0^{(1 + \xi(\theta)\rho)} \left( (\theta \cdot \bar \theta) \sin \phi\right) \sin^{n-1}\phi  \,d\phi \, d\theta\\
 & = \int_{S^{n-1}} (\theta \cdot \bar \theta) \int_0^{(1 + \xi(\theta)\rho)}\sin^{n}\phi  \,d\phi \, d\theta\\
 & = \int_{S^{n-1}} (\theta \cdot \bar \theta)\, F\left((1+\xi(\theta))\rho\right)   \, d\theta,
\end{align*}
where the first inner product is the standard Euclidean inner product, $F(t) = \int_0^t \sin(s)^{n}\,ds$ is  a  smooth hypergeometric function, and $\theta \cdot \bar \theta $ indicates the inner product with respect to the standard metric on $S^{n-1}$. A Taylor  expansion of $F$  and H\"{o}lder's inequality thus show that
\begin{align*}
  0 &=\int_{S^{n-1}} (\theta \cdot \bar \theta)\Big( F(\rho)  +  C_\rho \xi(\theta ) + o\left(\| \xi\|_{L^2(\pa \B_\rho)}\right)\Big) \, d\theta \\
& =  C_\rho  \int_{S^{n-1}} (\theta \cdot \bar \theta)\xi(\theta ) \, d\theta + o\left(\| \xi\|_{L^2(\pa \B_\rho)}\right).
  \end{align*}
So, we see that for $i-1,\dots, n$, we have $a_i = \int_{S^{n-1}}(\theta \cdot \bar \theta)\xi(\theta )) = o(\|\xi \|_{L^2(S^{n-1})})$ and conclude \eqref{e: constraints linearized}.
Define the functions 
\[
\hat{\xi} = \xi -a_0Y_0, \qquad 
\bar \xi = \hat{\xi} -\sum_{i=1}^{n}a_iY_i\,.
\]
 As a consequence of \eqref{e: constraints linearized}, we see that 
\begin{align*}
\| \xi \|_{H^{1/2}(\pa (\B_\rho))} &= \| \hat{\xi} \|_{H^{1/2}(\pa \B_\rho)} + o\left(\| \xi \|_{L^2(\pa \B_\rho)}^2\right) \\
&= |\bar{ \xi} \|_{H^{1/2}(\pa \B_\rho)} +o\left(\| \xi \|_{L^2(\pa \B_\rho)}^2\right)\,.
\end{align*}
	By Theorem~\ref{lem: second var and deficit} and applying Corollary~\ref{cor: second var controls h1/2} to $\bar{\xi}$, we have 
	\begin{align*}
		\lambda_1(\Omega) - \lambda_1(\B_\rho) & \geq \frac{1}{2}\delta^2\lambda(\B_\rho)[\hat{\xi},\hat{\xi}] +o\left( \| \xi \|_{H^{1/2}(\pa \B_\rho)}^2\right)\\
		& = \frac{1}{2}\delta^2\lambda(\B_\rho)[\bar{\xi},\bar{\xi}] + o\left( \| \xi \|_{H^{1/2}(\pa \B_\rho)}^2\right)\\
		& = \frac{\hat \eta}{2} \| \bar \xi\|_{H^{1/2}(\pa \B_\rho)}^2 +o\left( \| \xi \|_{H^{1/2}(\pa \B_\rho)}\right)\\
		&= \frac{\hat \eta}{2} \| \xi\|_{H^{1/2}(\pa \B_\rho)}^2 +o\left( \| \xi \|_{H^{1/2}(\pa \B_\rho)}\right).
	\end{align*}
Choosing $\e$ sufficiently small depending on $\bar \eta$ (and thus on $\gamma$ and $n$), we may absorb the second term on the right-hand side to conclude.
	\end{proof}

\bigskip
 
\subsubsection{The $H^{1/2}$ norm controls the distance }\label{ssec: h12 and distance}
Thanks to Theorem~\ref{th: deficit controls H1/2}, we will conclude the proof of Proposition~\ref{prop: quantitative for nearly spherical sets} as soon as we can show that the terms $|\Om \Delta \B_\rho|$ and $\int_{S^{n}}| u_\Om - u_{\B_\rho}|$ are controlled linearly by $\|\xi\|_{H^{1/2}(B)}.$ 
This is the content of the following proposition.

In the case of Euclidean space and hyperbolic space, the proof of this proposition will be identical up to replacing the coarea factor $\sin^{n-1} \phi$ with $\phi^{n-1}$ and $\sinh^{n-1}\phi$ respectively in various integrals throughout the proof.

\begin{proposition}\label{prop: control efunc} 
Fix $n\geq 2$ and $\g \in (0,\pi/2)$. There are positive constants $\e$ and $C$ depending on $n$ and $\g$ such that the following holds.  Let $|\pi/2-\rho| <\g.$ Let $\Omega \subset S^{n}$ be a nearly spherical set parameterized by $\xi$ 
	with $\|\xi\|_{C^{2,\alpha}(\pa \B_\rho)} \leq \e$  and $\lambda_1(\Omega)-\lambda_1(\B_\rho) \leq \e\lambda_1(\B_\rho).$ Then 
	\begin{align}\label{eqn: h1/2 controls asymmetry}
	|\Om \Delta \B_\rho|^2  +\int_{S^{n}}| u_\Om - u_{\B_\rho}|^2 & \leq  C\|\xi\|^2_{H^{1/2}(\pa \B_\rho)}.
	\end{align}
	\end{proposition}

\begin{proof}
The fact that $|\Omega\Delta \B_\rho| \leq C \|\xi\|_{L^2(\pa B)} \leq C\|\xi\|_{H^{1/2}(\pa \B_\rho)}$ is a standard computation; see \cite[Proof of Theorem 4.3]{CiLe12} or \cite[Lemma 4.2]{BDV15}.  Establishing $\int_{S^{n}}| u_\Om - u_{\B_\rho}|^2 \leq C\| \xi\|^2_{H^{1/2}(\pa \B_\rho)}$ is more involved and will be carried out in several steps.

{\it Step 1.}  First, we define a diffeomorphism between $\B_\rho$ and $\Om$ and express the equation solved by $u_\Om$ pulled back by this diffeomorphism. To this end, 
let $h : \B_\rho \to \R$ be the harmonic extension of $\xi$ in $\B_\rho$ defined in \eqref{eqn: harmonic extension}. Let  $\Psi: \B_\rho\to S^{n}$ be the smooth map defined by
\begin{equation}
	\Psi(x) = (1 + h(x))\exp_o(x),
\end{equation} 
where $\exp_o(x)$ is the exponential map at $o$.
In normal coordinates $\{x^\alpha\}$, the differential $d\Psi$ has coefficients 
\begin{equation}\label{eqn: coord expression}
\frac{\pa \Psi^\alpha}{\pa x^i} = (1+h(x)) \delta_i^\alpha +\frac{\pa}{\pa x^i} h(x) x^\alpha.	
\end{equation}
Since $h$ is harmonic and $|\na h|^2$ is subharmonic, the maximum principle ensures that 
$$
|h(x)|, |\na h(x)| \leq \| \xi\|_{C^{1}(\pa \B_\rho)} \leq C\| \xi \|_{C^{2,\alpha}(\pa \B_\rho)} \leq C\e	
$$
for all $x \in \B_\rho$, where $C=C(n,p)$. The metric coefficients $g_{ij}$ are uniformly bounded uniformly bounded above, so $|\frac{\pa }{\pa x^i} h(x) |\le C\e$ as well. 
We thus see that 
\begin{equation}\label{eqn: estimates on dPsi}
	\left|\frac{\pa \Psi^\alpha}{\pa x^i} - \delta^\alpha_i \right| \leq C\e,
\end{equation}
where $C$ is a constant depending only on $n$ and $\alpha$.
In particular, the differential $d\Psi$ is non-vanishing on $\B_\rho$, so  $\Psi$ is injective and hence a diffeomorphism onto its image. Since $\Psi$ maps $\pa \B_\rho$ to $\pa \Om,$ it follows that $\Om$ is the image of $\B_\rho$ under $\Psi$  and so $\Psi$ defines a smooth diffeomorphism from $\B_\rho$ to $\Omega.$ 
We let $\Phi: \Omega \to \B_\rho$ denote the inverse of $\Psi.$ 
\\

Recall that $u_\Omega$ denotes the first Dirichlet eigenfunction of $\Omega$,  normalized so that $\int u_\Omega^2 \,dx =1$ and extended by $0$ to be defined on all of $S^{n}$. We define $\hat u_\Om:\B_\rho\to \R$ to be the pullback $\hat u_\Omega(x) = u_\Omega(\Psi(x))$ of $u_\Omega$ by $\Psi.$  Then $\hat u_\Omega$ satisfies the equation 
\begin{equation}
	\begin{cases}
	-\mathcal{L}\hat u_\Omega = m\lambda_\Omega \hat u_\Omega & \text{ in } \B_\rho\\
	\hat u_\Om = 0& \text{ on }\pa \B_\rho.	
	\end{cases}
\end{equation}
Here and in the remainder of the proof we use the short-hand $\lambda_\Omega = \lambda_1(\Omega)$, and 
 $\mathcal{L}$ is the a linear divergence form operator given by 
\begin{equation}\label{eqn: pullback operator}
\mathcal{L} f = \,\divv (A\na f),  	
\end{equation}
where $A = m\mathcal{A}$ for $\mathcal{A}(x) 
 = d\Phi(\Psi(x)) (d\Phi(\Psi(x))^*$
and  $m = 1/\sqrt{\det \mathcal{A}}$.
From the definition of $\mathcal{L}$ and the coordinate expression \eqref{eqn: coord expression}, we see that the coefficients satisfy the pointwise estimates
\begin{equation}\label{eqn: L2 bound for coefficients euc}
	|A-\Id| \leq C(h + |\na h|), \qquad |m-1| \leq C(h + |\na h|)
\end{equation}
for all $x$ in $\B_\rho$, where $C$ is a dimensional constant. Furthermore, we have  
\begin{equation}\label{eqn:  holder coefficients a}
	\|A-\Id\|_{C^{0,\alpha}(\B_\rho)} \leq C\|h\|_{C^{1,\alpha}(\B_\rho)}, \qquad \|m-1\|_{C^{0,\alpha}(\B_\rho)} \leq C\|h\|_{C^{1,\alpha}(\B_\rho)}.
\end{equation}
In particular, by \eqref{eqn: estimates on dPsi} and by \cite[Corollary 8.35]{GT}, we have
\begin{equation}\label{eqn: C1alpha u}
	\| u_\Omega\|_{C^{1,\alpha}(\bar\Omega)}\leq C, \qquad 	\|\hat u_\Om\|_{C^{1,\alpha}(\bar \B_\rho)}\leq C.
\end{equation}
Additionally, we have 
\begin{equation}\label{eqn: L2 bound for coefficients}
	\| A-\Id\|_{L^2(\B_\rho)}\leq C\|h\|_{W^{1,2}(\B_\rho)}\leq C\| \xi\|_{H^{1/2}(\B_\rho)} .
\end{equation}
\\
\medskip


{\it Step 2:} We now prove the following integral estimate for $u_\Om - \hat u_\Om$:
\begin{align}\label{eqn: u vs pullback}
	\int_{S^{n}} |u_\Omega -\hat u_\Omega |^2  \leq C \ \| \xi\|_{H^{1/2}(\pa B)}^2.
\end{align}
 We argue in three separate regions: $\B_\rho\cap \Omega$, $\Omega\setminus \B_\rho$, and $\B_\rho\setminus \Omega$.  First, given any $x \in \B_\rho\cap \Omega$,  we have
\begin{align*}
	|\hat u_{\Omega}(x) -u_\Omega(x)|& = |u_\Om(\Psi(x)) - u_{\Omega}(x)| \leq \| u_\Omega\|_{C^1(\Omega)}d(x, \Psi (x)) \leq  Ch(x). 
\end{align*}
The final two inequalities follow from \eqref{eqn: C1alpha u} and the definition of $\Psi$ respectively. So, squaring and integrating this pointwise estimate and applying Lemma~\ref{lem: l2 bound} (or more specifically, the analogous statement with $\lambda_\rho$ replaced by $0$), we have 
\begin{align}\label{eqn: u vs pullback on intersection}
	\int_{\B_\rho\cap \Omega} |u_\Omega -\hat u_\Omega |^2   \leq C \| \xi\|_{L^2(\pa \B_\rho)}^2.
\end{align}

Next, for any $x \in \Omega \setminus \B_\rho$, note that $\hat{u}_\Om= 0$ and so $|u_{\Omega}(x) - \hat{u}_\Om(x)| = u_\Omega(x)$. Thanks to the estimate \eqref{eqn: C1alpha u} and the fact that $u_\Omega$ vanishes on $\pa \Omega,$ we see that $u_\Om(x) \leq \frac{C}{\rho} d(x,\pa \Om)$. Thus, in spherical coordinates $(\theta,\phi)$, we have $u_\Om(\theta,\phi)\leq C(1+\xi(\theta)-\phi/\rho).$ Hence, using the coarea formula, we have
\begin{align*}
	\int_{\Omega\setminus \B_\rho} u_\Omega^2 
	& = \int_{\{\xi >0\}} \int_\rho^{(1+\xi)\rho} u(\theta, \phi)^2  \sin(\phi)^{n-1} \,d\phi \,d\theta \\
	&  \leq C\sin(\rho)^{n-1} \int_{\{\xi >0\}} \int_\rho^{(1+\xi)\rho}(1+\xi(\theta)-\phi/\rho)^2  \,d\phi \,d\theta\\
	& = C\sin(\rho)^{n-1} \rho \int_{\{\xi >0\}} \xi^2  d\theta \leq C\rho \| \xi\|_{L^2(\pa \B_\rho)}^2. 
\end{align*}
where the penultimate equality comes from the change of variable $1+\xi(\theta)-\phi/\rho=s$.
The analogous argument on $\B_\rho\setminus \Omega$ using $\hat{u}_{\Omega}$ in place of $u_\Omega$. Together, these estimates along with \eqref{eqn: u vs pullback on intersection} show \eqref{eqn: u vs pullback}.
\\


\medskip
{\it Step 3:} Finally, we show that 
\begin{equation}\label{eqn: hat u and uB}
	\int_{\B_\rho} |\hat u_\Omega - u_{\B_\rho}|^2 
	\leq C \| \xi\|_{H^{1/2}(\pa \B_\rho)}^2.
\end{equation}
Let us begin with the following simplification. Let $\alpha_\Omega = \int_{\B_\rho} \hat{u}_\Omega u_{\B_\rho}$. Then
\[
	\int_{\B_\rho} |\hat u_\Omega - u_{\B_\rho}|^2  \leq 4	\int_{\B_\rho} |\hat u_\Omega - \alpha_\Om u_{\B_\rho}|^2 ,
\]
since the left-hand side is equal to $2(1-\alpha_\Om^2)$ while the right-hand side is equal to $1-\alpha_\Omega^2.$ 
It therefore suffices to show that 
\begin{equation}\label{eqn: intermediate claim}
	\int_{\B_\rho} |\hat u_\Omega -\alpha_\Omega u_{\B_\rho}|^2 
	\leq C \| \xi\|_{H^{1/2}(\pa \B_\rho)}^2.
\end{equation}	
To this end, set $v = \hat u_\Omega - \alpha_\Omega u_B.$ The idea behind showing \eqref{eqn: intermediate claim} is that $v$ solves an equation with right-hand side controlled in terms of $\xi$. More specifically,
 we see that $v$ satisfies the equation
\begin{align*}
		-\mathcal{L} v &= -\mathcal{L} \hat u_\Omega + \alpha_\Omega \mathcal{L}u_{\B_\rho}\\
		&= m\lambda_\Omega\hat u_\Omega -\alpha_\Omega \lambda_\rho  u_{\B_\rho} +\alpha_\Omega (\mathcal{L}-\Delta) u_{\B_\rho}\\
		& = \lambda_\Omega v +\alpha_\Om(\lambda_\Om - \lambda_{\rho}) u_{\B_\rho} + \alpha_\Om(\mathcal{L}-\Delta) u_{\B_\rho} +(m-1)\lambda_\Omega\hat u_\Omega.
\end{align*}
We multiply this equation by $v$ and integrate over $\B_\rho$. On one hand, note that \eqref{eqn: L2 bound for coefficients euc} ensures that $A \geq (1-C\e) \Id$. So integrating by parts, we find that
\begin{equation}\label{eqn: lower bound}
	- \int_{\B_\rho} v\, \mathcal{L}v 
	\geq (1-C\e) \int_{\B_\rho}|\na v|^2.
\end{equation}
On the other hand, note that $\int_{\B_\rho} v u_{\B_\rho} = 0$ and $v$ vanishes on $\pa \B_\rho$, and thus $v$ satisfies the improved Poincar\'{e} inequality 
\begin{equation}\label{eqn: orthogonal}
\int_{\B_\rho} |\na v|^2  \geq \lambda_2 \int_{\B_\rho} v^2 
\end{equation}
where $\lambda_2 =\lambda_2(\B_\rho)$ is the second Dirichlet eigenvalue of $\B_\rho$. 
So, again using the orthogonality of $u_{\B_\rho}$ and $v$, we find that 
\begin{equation}\label{eqn: upper bound}
\begin{split}
	 - \int_B v\, \mathcal{L}v & = \int_{\B_\rho} v\{ \lambda_\Omega v +\alpha_\Om(\lambda_\Om - \lambda_{\rho}) u_{\B_\rho} + \alpha_\Om(\mathcal{L}-\Delta) u_{\B_\rho}+(m-1)\lambda_\Omega\hat u_\Omega\}\\
	 & =   \int_{\B_\rho} \{\lambda_\Omega v^2 + \alpha_\Omega (\mathcal{L}-\Delta)u_{\B_\rho} v+(m-1)\lambda_\Omega\hat u_\Omega v\}\, dx\\
	 &\leq \frac{\lambda_\Omega}{\lambda_2} \int_{\B_\rho} |\na v|^2 + \alpha_\Omega \int_{\B_\rho} (\mathcal{L}-\Delta)u_{\B_\rho} v + \lambda_\Omega\int (m-1)\hat u_\Omega v.
\end{split}	
\end{equation}
Provided $\e $ is sufficiently small with respect to the spectral gap $\lambda_1/\lambda_2$, we may combine \eqref{eqn: lower bound} and \eqref{eqn: upper bound} and absorb the first term on the right-hand side of \eqref{eqn: upper bound} to find 
\begin{equation}\label{eqn: gradient bounds}
	c\int_{\B_\rho}|\na v|^2 \leq  \alpha_\Omega \int_{\B_\rho} (\mathcal{L}-\Delta)u_{\B_\rho} v + \lambda_\Omega\int_{\B_\rho} (m-1)\hat u_\Omega v.
\end{equation}
 Now, we bound the two terms on the right-hand side of \eqref{eqn: gradient bounds} separately. For the first term, we integrate by parts and then apply the Cauchy-Schwarz inequality and the bound \eqref{eqn: L2 bound for coefficients} on the coefficients to find  
\begin{align*}
\alpha_\Omega \int_{\B_\rho} (\mathcal{L}-\Delta)u_{\B_\rho} v & 
=  \alpha_\Omega \int_{\B_\rho} \langle (A-\Id)\na u_{\B_\rho}, \na v\rangle \\
	& \leq  \frac{\eta\, \alpha_\Omega}{2} \int_{\B_\rho} |\na v|^2 + \frac{\alpha_\Omega}{2\eta}\|\na u_{\B_\rho}\|_{L^\infty(\B_\rho)} \int_{\B_\rho} |A-\Id|^2  \\
	& \leq \frac{\eta\, \alpha_\Om }{2} \int_{\B_\rho} |\na v|^2 + \frac{\alpha_\Om }{2\eta}\|\na u_{\B_\rho}\|_{L^\infty(\B_\rho)} \|\xi\|_{H^{1/2}(\pa \B_\rho)}^2.
\end{align*}
Here $\langle \cdot, \cdot \rangle $ denotes the inner product on $S^{n}$ with respect to the round metric.
Provided that $\eta $ is chosen to be sufficiently small, we can absorb the first term to the left-hand side of \eqref{eqn: gradient bounds}. In a similar way, for the second term on the right-hand side of \eqref{eqn: gradient bounds}, we have
\begin{align*}
	 \lambda_\Omega\int_{\B_\rho} (m-1)\hat u_\Omega v & \leq \frac{\lambda_\Omega}{\eta} \int_{\B_\rho} |m-1|^2 + \lambda_\Omega\eta \| \hat{u}_\Om\|_{L^\infty(\B_\rho)}^2 \int_{\B_\rho} v^2 \\
	 & \leq \frac{\lambda_\Omega}{\eta}\|\xi\|_{H^{1/2}(\pa \B_\rho)}^2 + \lambda_\Omega\eta C \int_{\B_\rho} v^2\\
	 & \leq \frac{\lambda_\Omega}{\eta}\|\xi\|_{H^{1/2}(\pa \B_\rho)}^2 + \frac{(1+\e) \lambda_1}{\lambda_2} \eta C \int_{\B_\rho} |\na v| ^2.
\end{align*} 
Choosing $\eta$ to be sufficiently small, we then find that 
\begin{equation}
\frac{1}{4} \int_{\B_\rho}|\na v|^2 \leq  \frac{\alpha_\Omega}{2\eta}\|\na u_{\B_\rho}\|_{L^\infty(\B_\rho)} \|\xi\|_{H^{1/2}(\pa \B_\rho)}^2.
\end{equation}
This, together with the Poincar\'{e} inequality on $\B_\rho$,  establishes \eqref{eqn: intermediate claim} and thus \eqref{eqn: hat u and uB}. We combine \eqref{eqn: hat u and uB} and \eqref{eqn: u vs pullback} to conclude the proof of the proposition.
\end{proof}
Finally we can prove Theorem~\ref{prop: quantitative for nearly spherical sets}.
\begin{proof}[Proof of Theorem~\ref{prop: quantitative for nearly spherical sets}]
	Together Theorem~\ref{th: deficit controls H1/2} and Proposition~\ref{prop: control efunc} directly imply the theorem when $\lambda(\Omega ) - \lambda (\B_\rho) \leq \e \lambda (\B_\rho)$, while the result holds trivially when  $\lambda(\Omega ) \geq (1+\e) \lambda (\B_\rho)$ by choosing the constant $c$ to be sufficiently small.
\end{proof}
\bigskip

\subsection{Torsional rigidity, the Kohler Jobin inequality, and Quantitative Stability for the Faber-Krahn inequality}\label{sec: SP}

Thus far, we have proven Theorem~\ref{thm: quantitative sperner restated} in the special case of nearly spherical sets. The remainder of the proof of Theorem~\ref{thm: quantitative sperner restated} relies on a selection principle argument, as described in the introduction. The proof of the selection principle is carried out in our companion paper \cite{AKN1}. We restate the main result there, stated only in the generality needed here, in Theorem~\ref{cor1.2 from reg paper} below. For regularity reasons discussed in \cite{AKN1}, the functional involved in the selection principle involves an additional, possibly unexpected term:  the torsional rigidity.

Given an open set $\Omega \subset S^{n}$, the {\it torsional rigidity} of $\Omega$ is defined by 
\begin{equation}\label{eqn:  torsion}
	\tor(\Omega) =
	\inf \left\{ \int_\Om \frac{1}{2} |\na u|^2 -\int_{\Omega}u : \ u \in W^{1,2}_0(\Omega)\right\} \,.
\end{equation}
Naturally, the same quantity can be defined for an open subset on any Riemannian manifold. We note that our sign convention for the torsional rigidity differs by a sign from the definition stated in other contexts such as \cite{KJ1, BrascoInvitation}. With our sign convention, $\rm{tor}(\Omega)\leq 0$. Note also that $\text{tor}(\Omega)$, like the first Dirichlet eigenvalue, is decreasing under set inclusion.
The infimum in $\tor(\Omega)$ is achieved by the  unique solution to the equation
\begin{equation}
	\begin{cases}
		-\Delta u = 1 & \text{ in } \Om \\
		u =0 & \text{ on } \pa \Om.
	\end{cases}
\end{equation}
Applying the  Polya-Szeg\"{o} inequality to any function in the minimization problem \eqref{eqn:  torsion}, one sees that balls minimize the torsional rigidity among all quasi-open sets of a fixed volume:
 \begin{equation}\label{eqn: st v}
 	\tor(\Omega) \geq \tor(\B) \qquad  \text{ if  }\  \vol(\Omega) = \vol(\B).
 \end{equation} 
Since the Polya-Szeg\"{o} principle holds on Euclidean space and hyperbolic space, balls also minimize the torsional rigidity among sets of a fixed volume in these spaces. 

Combined with the Faber-Krahn inequality, we see that balls uniquely minimize the functional $\lambda_1(\Omega) + \frak{T} \text{tor}(\Omega)$ for any $\frak{T}\geq 0$ among sets of a fixed volume. 
In our companion paper, we establish the following global-to-local stability theorem; see \cite[Corollary 1.2 and Remark 1.3]{AKN1}.
\begin{theorem}[Selection Principle]\label{cor1.2 from reg paper} Let $M$ denote the round sphere, Euclidean space, or hyperbolic space. Fix $n\geq 2$ and an interval $[v_0, v_1] \in (0, |M|)$. There exist $\frak{T}_0>0$ and $\e>0$ depending on $n, v_0,$ and $v_1$ such that the following holds. 

Suppose there exists a constant $c>0$ such that for all $v \in [v_0, v_1]$, $\frak{T}< \frak{T}_0 $ and for every $\e$-nearly spherical set parametrized over $\B$ where $|\B|=v$,  we have
	\begin{equation}\label{eqn: nss torsion}
c\left(  |\Omega \Delta \B|^2 + \int_{M} |u_\Omega - u_\B|^2   \right) \leq \lambda_1(\Omega) - \lambda_1(\B) + \frak{T}(\rm{tor}(\Omega) - \rm{tor}(\B))\,.
	\end{equation}
Then  for all open sets $\Omega$ with $|\Omega|= |\B|$,
\begin{equation}\label{eqn: final}
c\inf_{x \in M}\left(  |\Omega \Delta \B(x)|^2 + \int_{M} |u_\Omega - u_\B(x)|^2   \right) \leq \lambda_1(\Omega) - \lambda_1(\B) + \frak{T}(\rm{tor}(\Omega) - \rm{tor}(\B))\,,
\end{equation}
 where $B(x)$ is a ball of the same radius as $B$.
\end{theorem}

Although the form of Theorem~\ref{cor1.2 from reg paper} does not appear to be immediately compatible with proving Theorem~\ref{thm: quantitative sperner restated}, an essential point is that  the quantities 
\begin{equation}\label{2deltas}
\delta(\Omega) = \lambda_1(\Omega) - \lambda_1(\B), \qquad \hat{\delta}(\Omega) = \lambda_1(\Omega) - \lambda_1(\B) + \frak{T}\,\left(\text{tor}(\Omega) - \text{tor}(\B)\right)
\end{equation}
are equivalent.
 While it is immediate that $\delta(\Omega) \leq \hat{\delta}(\Omega)$, we must call upon a result of Kohler-Jobin, Theorem~\ref{thm: KJ inequality} below, to show that 
$ \hat{\delta}(\Omega) \leq C \delta(\Omega) $.

  Polya and Szeg\"{o} \cite{PSBook} conjectured  that among sets $\Omega\subset \mathbb{R}^n$ of a fixed torsional rigidity $\tor(\Omega)$, balls minimize the first Dirichlet eigenvalue $\lambda_1(\Omega).$  This conjecture was proven by Kohler-Jobin through the introduction of a new symmetrization technique in \cite{KJ1,KJ2} (see also \cite{KJ3}, as well as \cite{BrascoInvitation} for extensions to the $p$-Lapace energy). Fundamentally, the proof is underpinned only by the coarea formula and the isoperimetry of balls. Basic modifications of Kohler-Jobin's argument show that this fact, now known as the  Kohler-Jobin inequality, holds on all simply connected space forms. 

\begin{theorem}[Kohler-Jobin inequality on simply connected space forms]\label{thm: KJ inequality}
	Let $M$ denote the round sphere, Euclidean space, or hyperbolic space. Fix an open set $\Omega \subset M$ and let $\B^*_\Omega \subset M$ denote a geodesic ball with the unique radius such that $\tor(\Om) = \tor(\B^*_\Omega)$. Then
	\begin{equation}\label{eqn: KJ}
		\lambda_1(\Omega) \geq \lambda_1(\B^*_{\Omega}).
	\end{equation}
	Equality is achieved in \eqref{eqn: KJ} if and only if $\Omega$ is a ball of this radius.
\end{theorem}

Theorem~\ref{thm: KJ inequality} allows us to deduce the following comparison between the Faber-Krahn deficit and the torsional rigidity  deficit and in particular show that $\delta(\Om)$ and $\hat{\delta}(\Om)$ are comparable.
\begin{corollary}\label{cor: KJ deficit} Let $M$ denote the round sphere, Euclidean space, or  hyperbolic space and fix $0 < \bar v <\bar V < \vol(M)$. There exists a constant $C=C(\bar v, \bar V)>0$  so that for any  $\Omega\subset M$ with $\vol(\Omega) =v_0 \in [\bar v, \bar V]$,
	\begin{equation}\label{eqn: deficit comparison}
		C\left( \lambda_1( \Omega) - \lambda_1(\B)\right) \geq \rm{tor}(\Om)-\rm{tor}(\B)
	\end{equation}
	where $\B$ is a geodesic ball with $\vol(\B) = v_0.$ In particular, $\hat{\delta}(\Omega) \leq \tilde{C} \delta(\Omega)$, where $\tilde{C} = 1+ C\frak{T}.$ 
\end{corollary}
\begin{remark}
	{\rm 
	One may use the  scaling invariance of Euclidean space to more directly conclude Corollary~\ref{cor: KJ deficit} from  the Kohler-Jobin inequality; see \cite[Proposition~2.1]{BDV15}. For the sake of a unified approach, we simply include Euclidean space in Corollary~\ref{cor: KJ deficit}.
	}
\end{remark}
\begin{proof}[Proof of Corollary~\ref{cor: KJ deficit}]
	As in Theorem~\ref{thm: KJ inequality}, let $\B^*_\Omega$ denote a ball with the uniquely chosen radius such that $\tor(\Omega) = \tor(\B^*_\Omega).$ Theorem~\ref{thm: KJ inequality} implies that
		\begin{align*}\label{eqn: basic app of KJ}
		\lambda_1(\Omega) - \lambda_1(\B) = \left[ \lambda_1(\B^*_\Omega) - \lambda_1(\B)\right] + \left[ \lambda_1(\Omega) - \lambda_1(\B^*_\Omega)\right]\geq  \lambda_1(\B^*_\Omega) - \lambda_1(\B)\,,
	\end{align*}
and so to establish \eqref{eqn: deficit comparison}, it suffices to prove that $C(\lambda_1(\B^*_\Omega) - \lambda_1(\B)) \geq  \tor(\B^*_\Omega) -  \tor(\B) $. Noting that the radius of $\B^*_\Omega$ is smaller than that of $\B$, the corollary will be proven once we establish the following claim. 

{\it Claim: } Fix $0<\bar r <\bar R$, with $\bar R<\pi$ if $M= S^n$, and for each $R \in (\bar r , \bar R)$, define the functions 
\begin{align*}
	f_R(r)  &=\lambda_1(\B_r) - \lambda_1(\B_R)\,,\\
	g_R(r) &=  \tor(\B_r)-\tor(\B_R)\,.
\end{align*}
There is a constant $C = C(\bar r, \bar R)$ such that for all $r \in (0,R)$,
	\begin{equation}\label{eqn: f g}
	g_R(r) \leq C f_R(r)\,.
	\end{equation}
	
	We first establish \eqref{eqn: f g} with a constant depending on $R$. Notice that $\lim_{r\to 0} f_R(r) = +\infty$ and $g_R(0) = -\tor(\B_R)$. So, we may find $\e =\e(R)$ such that $g_R(r) \leq f_R(r) $ for all $r \in (0,\e]$. Next, observe that the functions $f_R$ and $g_R$ are strictly decreasing and differentiable. Define the constants
	\begin{align*}
	m& := \min \{ -f'_R(r) \, :\,  r \in [\e, R]\}>0, \\
	M &:= \max \{ -g'_R(r)  \, :\,  r \in [\e, R]\}<\infty\,.
	\end{align*}
By the fundamental theorem of calculus and $f(R) = g(R) = 0$, we see that
	\begin{align*}
		f_R(r) & = \int_r^R -f_R'(r) \, dr \geq m(R-r),\\
		g_R(r) & = \int_r^R -g'_R(r) \, dr \leq M(R-r),
	\end{align*}
	and so $g_R(r) \leq  f_R(r)\times M/m$ for $r \in [\e,R]$. Letting $C = \max\{ M/m, 1\}>0$, we see that \eqref{eqn: f g} holds for $C = C(R)$. 
	Finally, by compactness, it is clear that we may choose $C$ uniformly for all $R \in [r_0, R_0]$.	This completes the proof of the corollary.
\end{proof}

Together Theorem~\ref{prop: quantitative for nearly spherical sets}, Theorem~\ref{cor1.2 from reg paper}, and Corollary~\ref{cor: KJ deficit} lead to the proof of Theorem~\ref{thm: quantitative sperner restated}: 
\begin{proof}[Proof of Theorem~\ref{thm: quantitative sperner restated}]
Combining Theorem~\ref{prop: quantitative for nearly spherical sets}, we see that \eqref{eqn: nss torsion} holds for nearly spherical sets, and so by Theorem~\ref{cor1.2 from reg paper}, we see that \eqref{eqn: final} holds for all open sets $\Omega$ with $|\Omega| = |\B|$. Corollary~\ref{cor: KJ deficit} applied to the right-hand side of \eqref{eqn: final} concludes the proof.
\end{proof}

\section*{Acknowledgements} MA was partially supported by Simons Collaboration Grant ID 637757.
During the course of this work, RN was partially supported by the National Science Foundation under Grant No. DMS-1901427, as well as Grant No. DMS-1502632 ``RTG: Analysis on manifolds'' at Northwestern University and Grant No. DMS-1638352 at the Institute for Advanced Study. 
This project originated from discussions at the 2018 PCMI Summer Session on Harmonic Analysis.


\bibliographystyle{amsplain}
\bibliography{quantitativeacfref}

\end{document}